\documentclass[11pt]{amsart}
\usepackage{amssymb,amsmath,amsfonts,amsthm, fullpage}

\numberwithin{equation}{section}
\usepackage{mathrsfs}

\usepackage{color}

\usepackage{verbatim,moreverb}

\renewcommand{\AA}{\mathbb A}

\newcommand{\CC}{\mathbb C}

\newcommand{\FF}{\mathbb F}

\newcommand{\PP}{\mathbb P}
\newcommand{\QQ}{\mathbb Q}
\newcommand{\RR}{\mathbb R}

\newcommand{\ZZ}{\mathbb Z} 

\newcommand{\Zhat}{\widehat\ZZ}

\newcommand{\calA}{\mathcal A} 
\newcommand{\calC}{\mathcal C} 
\newcommand{\calG}{\mathcal G} 
 
\newcommand{\calI}{\mathcal I} 

\newcommand{\calJ}{\mathcal J} 
\newcommand{\calN}{\mathcal N} 
\newcommand{\calU}{\mathcal U}

\newcommand{\scrI}{\mathscr I} 
\newcommand{\calS}{\mathcal S}

\newcommand{\OO}{\mathcal O}

\newcommand{\s}{\mathfrak{sl}}

\def\Spec{\operatorname{Spec}} 
\def\tr{\operatorname{tr}}
 \def\Gal{\operatorname{Gal}}
\def \Frob {\operatorname{Frob}}

\def \GL {\operatorname{GL}}  
\def \PGL {\operatorname{PGL}}
\def \SL {\operatorname{SL}}
\def \PSL {\operatorname{PSL}}

\def\Aut{\operatorname{Aut}} \def\End{\operatorname{End}}

\newcommand{\Hom}{\operatorname{Hom}}

\def\cyc{{\operatorname{cyc}}}
\def\ab{{\operatorname{ab}}}

\def\bbar#1{\setbox0=\hbox{$#1$}\dimen0=.2\ht0 \kern\dimen0 \overline{\kern-\dimen0 #1}}
\newcommand{\Qbar}{{\overline{\mathbb Q}}} 
 
\newcommand{\kbar}{\bbar{k}} 
 
\newcommand{\FFbar}{\overline{\FF}} 

\newcommand{\defi}[1]{\textsf{#1}} 

\newtheorem{thm}{Theorem}[section]
\newtheorem{lemma}[thm]{Lemma}
\newtheorem{cor}[thm]{Corollary}
\newtheorem{prop}[thm]{Proposition}
\newtheorem{conj}[thm]{Conjecture} 

\theoremstyle{example}

\theoremstyle{definition}

\newtheorem{defn}[thm]{Definition}

\theoremstyle{remark}
\newtheorem{remark}[thm]{Remark}

\newenvironment{romanenum}{\hfill \begin{enumerate} }{\end{enumerate}}
\newenvironment{alphenum}{\hfill \begin{enumerate} }{\end{enumerate}}

\definecolor{webbrown}{rgb}{.6,0,0}
\usepackage[
        colorlinks,
        linkcolor=webbrown,  filecolor=webcolor,  citecolor=webbrown, 
        backref,
        pdfauthor={David Zywina}, 
]{hyperref}
\usepackage[alphabetic,backrefs,lite]{amsrefs} 

\begin{document}

\title[]{Possible indices for the Galois image of elliptic curves over $\QQ$}

\author{David Zywina}
\address{Department of Mathematics, Cornell University, Ithaca, NY 14853, USA}
\email{zywina@math.cornell.edu}
\urladdr{http://www.math.cornell.edu/~zywina}

\begin{abstract} 
For a non-CM elliptic curve $E/\QQ$, the Galois action on its torsion points can be expressed in terms of a Galois representation $\rho_E \colon \Gal_\QQ :=\Gal(\Qbar/\QQ)\to \GL_2(\Zhat)$.   A well-known theorem of Serre says that the image of $\rho_E$ is open and hence has finite index in $\GL_2(\Zhat)$.     We will study what indices are possible assuming that we are willing to exclude a finite number of possible $j$-invariants from consideration.    For example, we will show that there is a finite set $J$ of rational numbers such that if $E/\QQ$ is a non-CM elliptic curve with $j$-invariant not in $ J$ and with surjective mod $\ell$ representations for all $\ell >37$ (which conjecturally always holds), then the index $[\GL_2(\Zhat) : \rho_E(\Gal_\QQ)]$ lies in the set 
\[
\calI = \left\{\begin{array}{c}2, 4, 6, 8, 10, 12, 16, 20, 24, 30, 32, 36, 40, 48, 54, 60, 72, 84, 96, 108, 112,120, 144, \\192, {220}, {240},  288, 336, {360},  384, {504}, 576, 768, 864, 1152, 1200, 1296, 1536 \end{array}\right\}.
\]
Moreover, $\calI$ is the minimal set with this property.
\end{abstract}

\maketitle
\section{Introduction} \label{S:intro}

\subsection{Main results}  \label{SS:main results}
Let $E$ be an elliptic curve defined over $\QQ$.  For each integer $N>1$, let $E[N]$ be the $N$-torsion subgroup of $E(\Qbar)$.  The group $E[N]$ is a free $\ZZ/N\ZZ$-module of rank 2 and has natural action of the absolute Galois group $\Gal_\QQ:= \Gal(\Qbar/\QQ)$.   This Galois action on $E[N]$ may be expressed in terms of a Galois representation
\[
\rho_{E,N} \colon \Gal_\QQ \to \Aut_{\ZZ/N\ZZ}(E[N]) \cong \GL_2(\ZZ/N\ZZ);
\]
it is uniquely determined up to conjugacy by an element of $\GL_2(\ZZ/N\ZZ)$.  
By choosing bases compatibly for all $N$, we may combine the representations $\rho_{E,N}$ to obtain a single Galois representation
\[
\rho_E\colon \Gal_\QQ \to \GL_2(\Zhat)
\]
that describes the Galois action on all the torsion points of $E$, where $\Zhat$ is the profinite completion of $\ZZ$.  If $E$ is \defi{non-CM}, then the following theorem of Serre \cite{MR0387283} says that the image is, up to finite index, as large as possible.

\begin{thm}[Serre] \label{T:Serre1} 
If $E/\QQ$ is a non-CM elliptic curve, then $\rho_E(\Gal_\QQ)$ has finite index in $\GL_2(\Zhat)$. 
\end{thm}

Serre's theorem is qualitative, and it natural to ask what the possible values for the index $[\GL_2(\Zhat):\rho_E(\Gal_\QQ)]$ are.  Our theorems address this question assuming that we are  willing to exclude a finite number of exceptional $j$-invariants from consideration; we will see later that the index $[\GL_2(\Zhat):\rho_E(\Gal_\QQ)]$ depends only on the $j$-invariant $j_E$ of $E$.

The most difficult part of Serre's proof of Theorem~\ref{T:Serre1} is to show that there is an integer $c_{E}$ such that $\rho_{E,\ell}(\Gal_\QQ)=\GL_2(\ZZ/\ell\ZZ)$ for all $\ell > c_{E}$.   In  \cite{MR0387283}*{\S4.3}, Serre asks whether one can choose $c_{E}$ independent of the elliptic curve (moreover, he asked whether this holds with $c_E=37$ \cite{MR644559}*{p.~399}).  We formulate this as a conjecture.

\begin{conj} \label{C:Serre1}
There is an absolute constant $c$ such that for every non-CM elliptic curve $E$ over $\QQ$, we have $\rho_{E,\ell}(\Gal_\QQ)=\GL_2(\ZZ/\ell\ZZ)$
for all $\ell> c$.
\end{conj}

Define the set 
\[
\calI:= \left\{\begin{array}{c}2, 4, 6, 8, 10, 12, 16, 20, 24, 30, 32, 36, 40, 48, 54, 60, 72, 84, 96, 108, 112,120, 144, \\192, {220}, {240},  288, 336, {360},  384, {504}, 576, 768, 864, 1152, 1200, 1296, 1536 \end{array}\right\}.
\]
\begin{thm} \label{T:main ineffective}
Fix an integer $c$.    There is a finite set $J$, depending only on $c$, such that if $E/\QQ$ is an elliptic curve with $j_E \notin J$ and $\rho_{E,\ell}$ surjective for all primes $\ell >c$, then  $[\GL_2(\Zhat) : \rho_E(\Gal_\QQ)]$ is an element of $\calI$.
\end{thm}

Assuming Conjecture~\ref{C:Serre1}, we can describe all possible indices $[\GL_2(\Zhat):\rho_E(\Gal_\QQ)]$ after first excluding elliptic curves with a finite number of exceptional $j$-invariants.

\begin{thm} \label{T:index Q}
Conjecture~\ref{C:Serre1} holds if and only if there exists a finite set $J\subseteq \QQ$ such that 
\[
[\GL_2(\Zhat):\rho_E(\Gal_\QQ)] \in \calI
\]
for every elliptic curve $E$ over $\QQ$ with $j_E\not\in J$.    
\end{thm}

For each integer $n\geq 1$, let $J_n$ be the set of $j\in \QQ$ that occur as the $j$-invariant of some elliptic curve $E$ over $\QQ$ with $[\GL_2(\Zhat):\rho_E(\Gal_\QQ)]=n$.    The following theorem shows that in Theorems~\ref{T:main ineffective} and \ref{T:index Q}, we cannot replace $\calI$ by a smaller set.

\begin{thm} \label{T:Jn}
For any integer $n\geq 1$, the set $J_n$ is infinite if and only if $n\in \calI$.
\end{thm}

\begin{remark} \label{R:main}
\begin{romanenum}
\item \label{R:main a}
Assuming Conjecture~\ref{C:Serre1},  Theorem~\ref{T:index Q} and Serre's theorem implies that there is an absolute constant $C$ such that $[\GL_2(\Zhat): \rho_{E}(\Gal_\QQ)] \leq C$ for all non-CM elliptic curves $E$ over $\QQ$.   

\item
The set $J$ in Theorem~\ref{T:index Q} contains more than the thirteen $j$-invariants coming from those elliptic curves over $\QQ$ with complex multiplication.    For example, the set $J$ contains $-7\cdot 11^3$ and $-7\cdot  137^3 \cdot 2083^3$ which arise from the two non-cuspidal rational points of $X_0(37)$, see \cite{MR0366930}.  If $E/\QQ$ is an elliptic curve with $j$-invariant $-7\cdot 11^3$ or $-7\cdot  137^3 \cdot 2083^3$, then one can show that $[\GL_2(\Zhat):\rho_E(\Gal_\QQ)] \geq 2736$.

\item
In our proofs of Theorems~\ref{T:main ineffective} and \ref{T:index Q}, the finite set $J$ that arises is ineffective.  The ineffectiveness arises from an application of Faltings' theorem to a finite number of modular curves of genus at least $2$.
\end{romanenum}
\end{remark}

\subsection{Overview}

In \S\ref{S:commutator}, we show that the index of $\rho_E(\Gal_\QQ)$ in $\GL_2(\Zhat)$ depends only on its commutator subgroup.     In \S\ref{S:modular curves}, we give some background on modular curves; for a fixed group $G$ of $\GL_2(\ZZ/N\ZZ)$ containing $-I$, its rational points will describe the elliptic curves $E/\QQ$ with $j_E\notin\{0,1728\}$ for which $\rho_{E,N}(\Gal_\QQ)$ is conjugate to a subgroup of $G$.  

In \S\ref{P:theoretic work}, we prove a version of Theorem~\ref{T:main ineffective} with $\calI$ replaced by another finite set $\scrI$ that is defined in terms of the congruence subgroups of $\SL_2(\ZZ)$ with genus $0$ or $1$.   Here we use Faltings' theorem to deal with rational points of several modular curves with genus at least $2$.

In \S\ref{S:index}, we describe how to compute the set $\scrI$; it agrees with our set $\calI$.  Here, and throughout the paper, we avoid computing models for modular curves.   For a genus $0$ modular curve, we use the Hasse principle to determine whether it is isomorphic to $\PP^1_\QQ$.   We compute the Jacobian of genus $1$ modular curves, up to isogeny, by counting their $\FF_p$-points via the moduli interpretation.  We also make use of the classification of genus $0$ and $1$ congruence subgroups due to Cummin and Pauli.

Finally, in \S\ref{S:proofs} we complete the proofs of Theorems~\ref{T:main ineffective}, \ref{T:index Q} and \ref{T:Jn}.

\subsection{Notation}
Fix a positive integer $m$.  Let $\ZZ_m$ be the ring that is the inverse limit of the rings $\ZZ/m^i\ZZ$ with respect to the reduction maps; equivalently, the inverse limit of $\ZZ/N\ZZ$, where $N$ divides some power of $m$.   We will make frequent use of the identifications $\ZZ_m= \prod_{\ell | m } \ZZ_\ell$ and $\Zhat = \prod_{\ell} \ZZ_\ell$, where $\ell$ denotes a prime.    In particular, $\ZZ_m$ depends only on the primes dividing $m$.

For a subgroup $G$ of $\GL_2(\ZZ/m\ZZ)$, $\GL_2(\ZZ_m)$ or $\GL_2(\Zhat)$ and an integer $N$ dividing $m$,  we denote by $G(N)$ the image of the group $G$ in $\GL_2(\ZZ/N\ZZ)$ under reduction modulo $N$.

All profinite groups will be considered with their profinite topologies.   The \emph{commutator subgroup} of a profinite group $G$ is the closed subgroup $G'$ generated by its commutators.

For each prime $p$, let $v_p \colon \QQ^\times \twoheadrightarrow \ZZ$ be the $p$-adic valuation.

\subsection*{Acknowledgments}
Thanks to Andrew Sutherland and David Zureick-Brown.    We have made use of some of the \texttt{Magma} code from \cite{Drew-Galois}.

The computations in \S\ref{S:index} were performed using the \texttt{Magma} computer algebra system \cite{Magma}; code can be found at \url{https://github.com/davidzywina/PossibleIndices}

\section{The commutator subgroup of the image of Galois} \label{S:commutator}

Let $E$ be a non-CM elliptic curve defined over $\QQ$.   Using the Weil pairing on the groups $E[N]$, one can show that the homomorphism $\det\circ \rho_E \colon \Gal_\QQ \to \Zhat^\times$ is equal to the cyclotomic character $\chi$.   Recall that $\chi\colon \Gal_\QQ \to \Zhat^\times$ satisfies $\sigma(\zeta)=\zeta^{\chi(\sigma) \bmod{n}}$ for any integer $n\geq 1$, where $\zeta \in \Qbar$ is an $n$-th root of unity and $\sigma\in \Gal_\QQ$.  

We first show that index of $\rho_E(\Gal_\QQ)$ in $\GL_2(\Zhat)$ is determined by its commutator subgroup. 

\begin{prop} \label{P:EC equality}
We have $[\GL_2(\Zhat):\rho_E(\Gal_\QQ)]  = [\SL_2(\Zhat)\colon \rho_E(\Gal_{\QQ})']$.
\end{prop}
\begin{proof}    The character $\chi$ is surjective, so $\det(\rho_E(\Gal_\QQ))=\Zhat^\times$ and hence $\rho_E(\Gal_\QQ)\cap \SL_2(\Zhat) = \rho_E(\Gal_{\QQ^\cyc})$, where $\QQ^\cyc$ is the cyclotomic extension of $\QQ$.  We thus have 
\[
[\GL_2(\Zhat):\rho_E(\Gal_\QQ)]  = [\SL_2(\Zhat):\rho_E(\Gal_\QQ)\cap \SL_2(\Zhat)]= [\SL_2(\Zhat)\colon \rho_E(\Gal_{\QQ^\cyc})].
\]   
It thus suffices to show that $\rho_E(\Gal_{\QQ^\cyc})$ equals $\rho_E(\Gal_{\QQ^\ab})=\rho_E(\Gal_\QQ)'$, where $\QQ^\ab \subseteq \Qbar$ is the maximal abelian extension of $\QQ$.  This follows from the Kronecker-Weber theorem which says that $\QQ^{\cyc}=\QQ^{\ab}$.
\end{proof}
  
\begin{remark}
\begin{romanenum}
\item
One can show that there are infinitely many different groups of the form $\rho_E(\Gal_\QQ)$ as $E$ varies over non-CM elliptic curves over $\QQ$; moreover, there are infinitely many such groups with index $2$ in $\GL_2(\Zhat)$.   One consequence of Proposition~\ref{P:EC equality} is that to compute the index $[\GL_2(\Zhat):\rho_E(\Gal_\QQ)]$ one does {not} need to know the full group $\rho_E(\Gal_\QQ)$, only $\rho_E(\Gal_\QQ)'$.

Conjecturally, there are only a finite number of subgroups of $\SL_2(\Zhat)$ of the form $\rho_E(\Gal_\QQ)'$ with a non-CM $E/\QQ$.   Indeed, suppose that Conjecture~\ref{C:Serre1} holds.      Remark~\ref{R:main}(\ref{R:main a}) and Proposition~\ref{P:EC equality} implies that the index of $[\SL_2(\Zhat): \rho_E(\Gal_\QQ)']$ is uniformly bounded for non-CM $E/\QQ$.   The finite number of possible groups of the form $\rho_E(\Gal_\QQ)'$ follows from their only being finitely many open subgroup of $\SL_2(\Zhat)$ of a given index.

\item
For a non-CM elliptic curve $E$ over a number field $K$, a similar argument shows that
\[
[\GL_2(\Zhat):\rho_E(\Gal_K)]  \leq [\Zhat^\times: \chi(\Gal_K)]\cdot [\SL_2(\Zhat)\colon \rho_E(\Gal_{K})'].
\]
The inequality may be strict if $K\neq \QQ$ (the cyclotomic extension of $K$ does not agree with the maximal abelian extension of $K$).
\end{romanenum}
\end{remark}

The following corollary show that for an elliptic curve $E/\QQ$, the index of $\rho_E(\Gal_\QQ)$ in $\GL_2(\Zhat)$ depends only on the $\Qbar$-isomorphism class of $E$.  In particular, the $j$-invariant is the correct notion to use in Theorems~\ref{T:index Q} and \ref{T:Jn}.  

\begin{cor} \label{C:index depends on j}
For an elliptic curve $E$ over $\QQ$, the index $[\GL_2(\Zhat): \rho_{E}(\Gal_\QQ)]$ depends only on the $j$-invariant of $E$.
\end{cor}
\begin{proof}
Suppose that $E_1$ and $E_2$ are elliptic curves over $\QQ$ with the same $j$-invariant (and hence isomorphic over $\Qbar$).  If $E_1$ (and hence $E_2)$ has complex multiplication, then both indices are infinite.   We may thus assume that $E_1$ and $E_2$ are non-CM.  Since they have the same $j$-invariant, $E_1$ and $E_2$ are isomorphic over a quadratic extension $L$ of $\QQ$.  Fixing such an isomorphism, we can identify the representations $\rho_{E_1}|_{\Gal_L}$ and $\rho_{E_2}|_{\Gal_L}$.  We have $L\subseteq \QQ^\ab$, so the groups $\rho_{E_1}(\Gal_{\QQ^\ab}) = \rho_{E_1}(\Gal_\QQ)'$ and $\rho_{E_2}(\Gal_{\QQ^\ab}) = \rho_{E_2}(\Gal_\QQ)'$ are equal under this identification.   The corollary then follows immediately from Proposition~\ref{P:EC equality}.
\end{proof}

\section{Modular curves} \label{S:modular curves}

Fix a positive integer $N$ and a subgroup $G$ of $\GL_2(\ZZ/N\ZZ)$ containing $-I$ that satisfies $\det(G)=(\ZZ/N\ZZ)^\times$.  Denote by $Y_G$ and $X_G$, the $\ZZ[1/N]$-schemes that are the coarse space of the algebraic stacks $\mathscr{M}_G^\circ[1/N]$ and $\mathscr{M}_G[1/N]$, respectively, from \cite{MR0337993}*{IV~\S3}.     We refer to \cite{MR0337993}*{IV} for further details.

The $\ZZ[1/N]$-scheme $X_G$ is smooth and proper and $Y_G$ is an open subscheme of $X_G$.  The complement of $Y_G$ in $X_G$, which we denote by $X_G^\infty$, is a finite \'etale scheme over $\ZZ[1/N]$, see \cite{MR0337993}*{IV~\S5.2}.  The fibers of $X_G$ are geometrically irreducible, see \cite{MR0337993}*{IV~Corollaire~5.6}; this uses our assumption that $\det(G)=(\ZZ/N\ZZ)^\times$.

In later sections, we will mostly work with the generic fiber of $X_G$, which we will also denote by $X_G$, which is a smooth, projective and geometrically irreducible curve over $\QQ$ (similarly, we will work with the generic fiber of $Y_G$ which will be a non-empty open subvariety of $X_G$).\\

Fix a field $k$ whose characteristic does not divide $N$; for simplicity, we will also assume that $k$ is perfect.  Choose an algebraic closure $\kbar$ of $k$ and set $\Gal_k := \Gal(\kbar/k)$.

In \S\ref{SS:points of $Y_G$}, we use the moduli property of $\mathscr{M}_G^\circ[1/N]$ to give a description of the sets $Y_G(k)$ and $Y_G(\kbar)$.  In \S\ref{SS:j-line}, we describe the natural morphism from $Y_G$ to the $j$-line.   In \S\ref{SS:cusps}, we give a way to compute the cardinality of the finite set $X_G^\infty(k)$ of \emph{cusps} of $X_G$ that are defined over $k$.  In \S\ref{SS:real points}, we determine when the set $Y_G(\RR)$ is non-empty.  In \S\ref{SS:complex points}, we will observe that $Y_G(\CC)$ as a Riemann surface is isomorphic to the quotient of the upper-half plane by the congruence subgroup $\Gamma_G$ consisting of $A\in \SL_2(\ZZ)$ for which $A$ modulo $N$ lies $G$.   Finally in \S\ref{SS:Fp points}, we explain how to compute the cardinality of $X_G(\FF_p)$ for primes $p\nmid 6N$.

\subsection{Points of $Y_G$} \label{SS:points of $Y_G$}

For an elliptic curve $E$ over $\kbar$, let $E[N]$ be the $N$-torsion subgroup of $E(\kbar)$.  A \emph{$G$-level structure} for $E$ is an equivalence class $[\alpha]_G$ of group isomorphisms $\alpha\colon E[N] \xrightarrow{\sim} (\ZZ/N\ZZ)^2$, where we say that $\alpha$ and $\alpha'$ are equivalent if $\alpha= g\circ \alpha'$ for some $g\in G$.   We say that two pairs $(E,[\alpha]_G)$ and $(E',[\alpha']_G)$, both consisting of an elliptic curve over $\kbar$ and a $G$-level structure, are \emph{isomorphic} if there is an isomorphism $\phi\colon E\to E'$ of elliptic curves such that $[\alpha]_G=[\alpha' \circ \phi]_G$, where we also denote by $\phi$ the isomorphism $ E[N]\to E'[N]$, $P\mapsto \phi(P)$.

From \cite{MR0337993}*{IV~Definition~3.2},  $\mathscr{M}_G^\circ[1/N](\kbar)$ is the category with objects $(E,[\alpha]_G)$, i.e., elliptic curves over $\kbar$ with a $G$-level structure, and morphisms being the isomorphisms between such pairs.   Since $Y_G$ is the coarse space of $\mathscr{M}_G^\circ[1/N]$, we find that $Y_G(\kbar)$ is the set of isomorphisms classes in $\mathscr{M}_G^\circ[1/N](\kbar)$.\\

The functoriality of $\mathscr{M}_G^\circ[1/N]$, gives an action of the group $\Gal_k$ on $Y_G(\kbar)$.    Take any $\sigma \in \Gal_k$.    Let $E^\sigma$ be the base extension of $E/\kbar$ by the morphism $\Spec \kbar \to \Spec \kbar$ coming from $\sigma$.    The natural morphism $E^\sigma \to E$ of schemes induces a group isomorphism $E^\sigma[N] \to E[N]$ which, by abuse of notation, we will denote by $\sigma^{-1}$.     More explicitly, if $E$ is given by a Weierstrass equation $y^2+a_1xy+a_3y = x^3 +a_4 x + a_6$ with $a_i \in \kbar$, we may take $E^\sigma$ to be the curve defined by $y^2+\sigma(a_1)xy+\sigma(a_3)y = x^3 +\sigma(a_4) x + \sigma(a_6)$; the isomorphism $E^\sigma[N] \to E[N]$ is then given by $(x,y)\mapsto (\sigma^{-1}(x),\sigma^{-1}(y))$.    For a point $P \in Y_G(\kbar)$ represented by a pair $(E,[\alpha]_G)$, the point $\sigma(P) \in Y_G(\kbar)$ is represented by $(E^\sigma, [\alpha \circ \sigma^{-1}]_G)$.\\

Since $k$ is perfect, $Y_G(k)$ is the subset of $Y_G(\kbar)$ stable under the action of $\Gal_k$.   The following lemma describes $Y_G(k)$.   For an elliptic curve $E$ over $k$, let $E[N]$ be the $N$-torsion subgroup of $E(\kbar)$.   Each $\sigma\in \Gal_k$ gives an isomorphism $E[N]\xrightarrow{\sim} E[N]$, $P\mapsto \sigma^{-1}(P)$ that we will also denote by $\sigma^{-1}$.    

\begin{lemma}  \label{L:YG(k) meaning}
\begin{romanenum}
\item \label{L:YG(k) meaning i}
Every point $P \in Y_G(k)$ is represented by a pair $(E,[\alpha]_G)$ with $E$ defined over $k$.
\item \label{L:YG(k) meaning ii}
Let $P \in Y_G(\kbar)$ be a point represented by a pair $(E,[\alpha]_G)$ with $E$ defined over $k$.   Then $P$ is an element of $Y_G(k)$ if and only if for all $\sigma \in \Gal_k$, we have an equality
\[
\alpha \circ \sigma^{-1} = g\circ \alpha \circ \phi
\]
of isomorphisms $E[N]\xrightarrow{\sim} (\ZZ/N\ZZ)^2$ for some $\phi\in \Aut(E_{\kbar})$ and $g\in G$.
\end{romanenum}
\end{lemma}
\begin{proof}
First suppose that $(E,[\alpha]_G)$ represents a point $P \in Y_G(k)$.   To prove (\ref{L:YG(k) meaning i}) it suffices to show that $E$ is isomorphic over $\kbar$ to an elliptic curve defined over $k$.   So we need only show that $j_E$ is an element of $k$.   For any $\sigma \in \Gal_k$, the point $P = \sigma(P)$ is also represented by $(E^\sigma, [\alpha\circ \sigma^{-1}]_G)$.  This implies that $E$ and $E^\sigma$ are isomorphic and hence $\sigma(j_E) = j_E$.   We thus have $j_E \in k$ since $k$ is perfect.   

We now prove (\ref{L:YG(k) meaning ii}).  Let $P \in Y_G(\kbar)$ be a point represented by a pair $(E,[\alpha]_G)$ with $E$ defined over $k$.   Take any $\sigma \in \Gal_k$.    The point $\sigma(P)$ is represented by $(E,[\alpha\circ \sigma^{-1}]_G)$; we can make the identification $E=E^\sigma$ since $E$ is defined over $k$.    We have $\sigma(P)=P$ if and only if there is an automorphism $\phi \in \Aut(E_{\kbar})$ such that $[\alpha \circ \sigma^{-1}]_G = [\alpha \circ \phi]_G$.    Since $k$ is perfect, we have $P\in Y_G(k)$ if and only if for all $\sigma \in \Gal_k$, we have $[\alpha \circ \sigma^{-1}]_G = [\alpha \circ \phi]_G$ for some $\phi \in \Aut(E_{\kbar})$; this is a reformulation of part (\ref{L:YG(k) meaning ii}).
\end{proof}

\subsection{Morphism to the $j$-line} \label{SS:j-line}
If $G=\GL_2(\ZZ/N\ZZ)$, then there is only a single $G$-level structure for each elliptic curve.   There is an isomorphism $Y_{\GL_2(\ZZ/N\ZZ)} = \AA^1_{\ZZ[1/N]}$; on $\kbar$-points, it takes a point represented by a pair $(E,[\alpha]_G)$ to the $j$-invariant $j_E \in \kbar$.    

If $G'$ is a subgroup of $\GL_2(\ZZ/N\ZZ)$ containing $G$, then there is a natural morphism $Y_G \to Y_{G'}$.   In particular, $G'=\GL_2(\ZZ/N\ZZ)$ gives a morphism
\[
\pi_G \colon Y_G \to \AA^1_{\ZZ[1/N]}
\]
that maps a $\kbar$-point represented by a pair $(E,[\alpha]_G)$ to the $j$-invariant of $E$.\\

Fix an elliptic curve $E$ over $k$.  By choosing a basis for $E[N]$ as a $\ZZ/N\ZZ$-module, the Galois action on $E[N]$ can be expressed in terms of a representation $\rho_{E,N} \colon \Gal_k \to \GL_2(\ZZ/N\ZZ)$; this is the same as the earlier definition with $k=\QQ$.   The representation $\rho_{E,N}$ is uniquely determined up to conjugation by an element of $\GL_2(\ZZ/N\ZZ)$.

\begin{prop} \label{P:key}
Let $E$ be an elliptic curve over $k$ with $j_E \notin \{0,1728\}$.  The group $\rho_{E,N}(\Gal_k)$ is conjugate in $\GL_2(\ZZ/N\ZZ)$ to a subgroup of $G$ if and only if $j_E$ is an element of $\pi_G(Y_G(k))$.
\end{prop}
\begin{proof}
First suppose that $\rho_{E,N}(\Gal_k)$ is conjugate to a subgroup of $G$.  There is thus an isomorphism $\alpha \colon E[N] \xrightarrow{\sim} (\ZZ/N\ZZ)^2$ such that $\alpha \circ \sigma \circ \alpha^{-1} \in G$ for all $\sigma\in \Gal_k$.  By Lemma~\ref{L:YG(k) meaning}(\ref{L:YG(k) meaning ii}), with $\phi=1$, the pair $(E,[\alpha]_G)$ represents a point $P\in Y_G(k)$.   Therefore, $j_E = \pi_G(P)$ is an element of $\pi_G(Y_G(k))$.

Now suppose that $j_E = \pi_G(P)$ for some point $P \in Y_G(k)$.  Lemma~\ref{L:YG(k) meaning} implies that $P$ is represented by a pair $(E,[\alpha]_G)$, where for all $\sigma\in \Gal_k$, we have $\alpha\circ \sigma^{-1} \circ \phi \circ \alpha^{-1} \in G$ for some automorphism $\phi$ of $E_{\kbar}$.   The assumption $j_E \notin \{0,1728\}$ implies that $\Aut(E_{\kbar}) = \{\pm 1\}$.   In particular, every automorphism of $E_{\kbar}$ acts on $E[N]$ as $\pm I$.   Since $G$ contains $-I$, we deduce that  $\alpha\circ \sigma^{-1} \circ \alpha^{-1} \in G$ for all $\sigma \in \Gal_k$.   We may choose $\rho_{E,N}$ so that $\rho_{E,N}(\sigma)=\alpha\circ \sigma \circ \alpha^{-1}$ for all $\sigma\in \Gal_k$, and hence $\rho_{E,N}(\Gal_k)$ is a subgroup of $G$.
\end{proof}

Take any $j\in k$ and fix an elliptic curve $E$ over $k$ with $j_E=j$.    Let $M$ be the group of isomorphisms $E[N] \xrightarrow{\sim} (\ZZ/N\ZZ)^2$.   Composition gives an action of the groups  $G$ and $\Aut(E_{\kbar})$ on $M$; they are left and right actions, respectively.   The map $\alpha \in M \mapsto (E,[\alpha]_G)$ induces a bijection
\begin{align} \label{E:bijection double coset}
G\backslash M /\Aut(E_{\kbar}) \xrightarrow{\sim} \{ P \in Y_G(\kbar) : \pi_G(P)=j\}. 
\end{align}
The group $\Gal_k$ acts on $M$ by the map $\Gal_k \times M \to M$, $(\sigma,\alpha)\mapsto \alpha \circ \sigma^{-1}$.     From the description of the Galois action in \S\ref{SS:points of $Y_G$}, we find that the bijection (\ref{E:bijection double coset}) respects the $\Gal_k$-actions.   The following lemma is now immediate (again we are using that $k$ is perfect).

\begin{lemma} \label{L:fiber cardinality}
The set $\{P \in Y_G(k): \pi_G(P) = j\}$ has the same cardinality as the subset of $G\backslash M /\Aut(E_{\kbar})$ fixed by the $\Gal_k$-action.
\end{lemma}

\subsection{Cusps} \label{SS:cusps}

In this section, we state an analogue of Lemma~\ref{L:fiber cardinality} for $X_G^\infty(k)$.       Let $M$ be the group of isomorphisms  $\mu_N \times \ZZ/N\ZZ  \xrightarrow{\sim} (\ZZ/N\ZZ)^2$, where $\mu_N$ is the group of $N$-th roots of unity in $\kbar$.   The group $\Gal_k$ acts on $M$ by the map $\Gal_k \times M \to M$, $(\sigma,\alpha)\mapsto \alpha \circ \sigma^{-1}$, where $\sigma^{-1}$ acts on $\mu_N$ as usual and trivially on $\ZZ/N\ZZ$.  Let $U$ be the subgroup of $\Aut(\mu_N\times \ZZ/N\ZZ)$ given by the matrices $\pm \left(\begin{smallmatrix}1 & u \\0 & 1\end{smallmatrix}\right)$ with $u\in \Hom(\ZZ/N\ZZ,\mu_N)$.   Composition gives an action of the groups  $G$ and $U$ on $M$; they are left and right actions, respectively.    Construction~5.3 of \cite{MR0337993}*{VI} shows that there is a bijection
\[
X_G^\infty(\kbar) \xrightarrow{\sim} G\backslash M / U
\]
that respects the actions of $\Gal_k$.    We thus have a bijection between $X_G^\infty(k)$ and the subset of $G\backslash M /U$ fixed by the action of $\Gal_k$.

Observe that the cardinality of $X_G^\infty(k)$ depends only on $G$ and the image of the character $\chi_N \colon \Gal_k \to (\ZZ/N\ZZ)^\times$ describing the Galois action on $\mu_N$, i.e., $\sigma(\zeta)=\zeta^{\chi_N(\sigma)}$ for all $\sigma \in \Gal_k$ and all $\zeta\in \mu_N$.    Let $B$ be the subgroup of $\GL_2(\ZZ/N\ZZ)$ consisting of matrices of the form $\left(\begin{smallmatrix}b & 0 \\0 & 1\end{smallmatrix}\right)$ with $b\in \chi_N(\Gal_k)$.    Let $U$ be the subgroup of $\GL_2(\ZZ/N\ZZ)$ generated by $-I$ and $\left(\begin{smallmatrix}1 & 1 \\0 & 1\end{smallmatrix}\right)$.       The group $B$ normalizes $U$ and hence right multiplication gives a well-defined action of $B$ on $G\backslash\GL_2(\ZZ/N\ZZ) / U$.   The following lemma is now immediate.

\begin{lemma} \label{L:cusp bijection}
The set $X_G^\infty(k)$ has the same cardinality as the subset of $G\backslash \GL_2(\ZZ/N\ZZ) / U$ fixed by right multiplication by $B$.
\end{lemma}

\subsection{Real points} \label{SS:real points}

The following proposition tells us when $Y_G(\RR)$ is non-empty.

\begin{prop} \label{P:R-points for YG}
The set $Y_G(\RR)$ is non-empty if and only if $G$ contains an element that is conjugate in $\GL_2(\ZZ/N\ZZ)$ to $\left(\begin{smallmatrix}1 & 0 \\0 & -1\end{smallmatrix}\right)$ or $\left(\begin{smallmatrix}1 & 1 \\0 & -1\end{smallmatrix}\right)$.
\end{prop}
\begin{proof}
Let $E$ be any elliptic curve over $\RR$.   As a topological group, the identity component of $E(\RR)$ is isomorphic to $\RR/\ZZ$.   So there is a point $P_1 \in E(\RR)$ of order $N$.  Choose a second point $P_2\in E(\CC)$ so that $\{P_1,P_2\}$ is a basis of $E[N]$ as a $\ZZ/N\ZZ$-module.   Define $\rho_{E,N}$ with respect to this basis.

Let $\sigma\in \Aut(\CC/\RR)$ be the complex conjugation automorphism.  We have $\sigma(P_1)=P_1$ and $\sigma(P_2)=bP_1 + d P_2$ for some $b,d \in \ZZ/N\ZZ$, i.e.,  $\rho_{E,N}(\sigma) := \left(\begin{smallmatrix}1 & b \\0 & d\end{smallmatrix}\right) \in \GL_2(\ZZ/N\ZZ)$.   Using the Weil pairing, we find that $\det(\rho_{E,N}(\sigma))$ describes how $\sigma$ acts on the $N$-th roots of unity.   Since complex conjugation inverts roots of unity, we have $\det(\rho_{E,N}(\sigma))=-1$ and hence $d=-1$.   For a fixed $m \in \ZZ/N\ZZ$, define points $P_1':=P_1$ and $P_2' :=P_2+m P_1$.  The points $\{P_1',P_2'\}$ are a basis for $E[N]$, and we have $\sigma(P_1')=P_1'$ and 
\[
\sigma(P_2') = (bP_1-P_2) + m P_1 = -(P_2+mP_1) + (b+2m ) P_1= - P_2' + (b+2m) P_1'.
\]
We can choose $m$ so that $b+2m$ is congruent to $0$ or $1$ modulo $N$;  with such an $m$ and the choice of basis $\{P_1',P_2'\}$, the matrix $\rho_{E,N}(\sigma)$ will be $\left(\begin{smallmatrix}1 & 0 \\0 & -1\end{smallmatrix}\right)$ or $\left(\begin{smallmatrix}1 & 1 \\0 & -1\end{smallmatrix}\right)$.

We claim that both of the matrices $\left(\begin{smallmatrix}1 & 0 \\0 & -1\end{smallmatrix}\right)$ and $\left(\begin{smallmatrix}1 & 1 \\0 & -1\end{smallmatrix}\right)$ are conjugate to $\rho_{E,N}(\sigma)$ for some $E/\RR$ with $j_E \notin \{0,1728\}$.   This is clear if $N$ is odd since the two matrices are then conjugate (we could have solved for $m$ in either of the congruences above).     If $N$ is even, then it suffices to show that both possibilities occur when $N=2$; this is easy (if $E/\QQ$ is given by a Weierstrass equation $y^2=x^3+ax+b$, the two possibilities are distinguished by the number of real roots that $x^3+ax+b$ has).   

Using Proposition~\ref{P:key}, we deduce that $\pi_G(Y_G(\RR))-\{0,1728\}$ is non-empty if and only if $G$ contains an element that is conjugate in $\GL_2(\ZZ/N\ZZ)$ to $\left(\begin{smallmatrix}1 & 0 \\0 & -1\end{smallmatrix}\right)$ or $\left(\begin{smallmatrix}1 & 1 \\0 & -1\end{smallmatrix}\right)$.   To complete the proof of the proposition, we need to show that if $\pi_G(Y_G(\RR)) \subseteq \{0,1728\}$, then $\pi_G(Y_G(\RR))$ is empty.    So suppose that $\pi_G(Y_G(\RR)) \subseteq \{0,1728\}$ and hence $Y_G(\RR)$ is finite.   However, since $Y_G$ over $\QQ$ is a smooth, geometrically irreducible curve, the set $Y_G(\RR)$ is either empty or infinite.
\end{proof}

\subsection{Complex points} \label{SS:complex points}
The complex points $Y_G(\CC)$ form a Riemann surface.   In this section, we describe it as a familiar quotient of the upper half plane by a congruence subgroup.  

Let $\mathfrak{H}$ be the complex upper half plane.   For $z\in \mathfrak{H}$ and $\gamma=\left(\begin{smallmatrix}a & b \\c & d\end{smallmatrix}\right) \in \SL_2(\ZZ)$, set $\gamma(z) := (az+b)/(cz+d)$.   We let $\SL_2(\ZZ)$ act on the \emph{right} of $\mathfrak{H}$ by $\mathfrak{H}\times \SL_2(\ZZ) \to \mathfrak{H}$, $(z,\gamma)\mapsto \gamma^t(z)$, where $\gamma^t$ is the transpose of $\gamma$.   For a congruence subgroup $\Gamma$, the quotient $\mathfrak{H}/\Gamma$ is a smooth Riemann surface.  

We define the \defi{genus} of a congruence subgroup $\Gamma$ to be the genus of the Riemann surface $\mathfrak{H}/\Gamma$.  

\begin{remark}
One could also consider the quotient $\Gamma\backslash \mathfrak{H}$ of $\mathfrak{H}$ under the left action given by $(\gamma,z)\mapsto \gamma(z)$;  it is isomorphic to the Riemann surface $\mathfrak{H}/\Gamma$ (use that $\gamma^t = B \gamma^{-1} B^{-1}$ for all $\gamma \in \Gamma$, where $B=\left(\begin{smallmatrix} 0 & -1 \\1 & 0\end{smallmatrix}\right)$ ).    In particular, the genus of $\Gamma\backslash \mathfrak{H}$ agrees with the genus of $\Gamma$.
\end{remark}

Let $\Gamma_G$ be the congruence subgroup consisting of matrices $\gamma \in \SL_2(\ZZ)$ whose image modulo $N$ lies in $G$.     The image of $\Gamma_G$ modulo $N$ is $G\cap \SL_2(\ZZ/N\ZZ)$ since the reduction map $\SL_2(\ZZ)\to \SL_2(\ZZ/N\ZZ)$ is surjective.  In particular, $\Gamma_G$ depends only on the group $G \cap\SL_2(\ZZ/N\ZZ)$ and we have
\[
[\SL_2(\ZZ): \Gamma_G]= [\SL_2(\ZZ/N\ZZ) : G \cap\SL_2(\ZZ/N\ZZ)].
\]

\begin{prop} \label{P:isomorphic Riemann surfaces}
The Riemann surfaces $Y_G(\CC)$ and $\mathfrak{H}/\Gamma_G$ are isomorphic.   In particular, the genus of $Y_G$ is equal to the genus of $\Gamma_G$.
\end{prop}
\begin{proof}
Set $X^{\pm}:=\CC-\RR$; we let $\GL_2(\ZZ)$ act on the right in the same manner $\SL_2(\ZZ)$ acts on $\mathfrak{H}$.   We also let $\GL_2(\ZZ)$ act on the right of $G\backslash \GL_2(\ZZ/N\ZZ)$ by right multiplication.  From \cite{MR0337993}*{IV~\S5.3}, we have an isomorphism
\[
Y_G(\CC) \cong  \big(X^\pm \times (G\backslash \GL_2(\ZZ/N\ZZ) )\big)/\GL_2(\ZZ).
\]
Using that $\det(G)=(\ZZ/N\ZZ)^\times$ and setting $H:= G \cap\SL_2(\ZZ/N\ZZ)$, we find that the natural maps  
\begin{align*}
(\mathfrak{H} \times (G\backslash \GL_2(\ZZ/N\ZZ) ))/\SL_2(\ZZ) &\to (X^\pm \times (G\backslash \GL_2(\ZZ/N\ZZ) ))/\GL_2(\ZZ) \quad \text{and} \\
(\mathfrak{H} \times (H\backslash \SL_2(\ZZ/N\ZZ) ))/\SL_2(\ZZ) &\to (\mathfrak{H} \times (G\backslash \GL_2(\ZZ/N\ZZ) ))/\SL_2(\ZZ)
\end{align*}
are isomorphisms of Riemann surfaces.  It thus suffices to show that $\mathfrak{H}/\Gamma_G$ and $ (\mathfrak{H} \times (H\backslash \SL_2(\ZZ/N\ZZ) ))/\SL_2(\ZZ)$ are isomorphic.  Define the map
\[
\varphi\colon \mathfrak{H}/\Gamma_G \to (\mathfrak{H} \times (H\backslash \SL_2(\ZZ/N\ZZ) ))/\SL_2(\ZZ)
\]
that takes a class containing $z$ to the class represented by $(z, H\cdot I)$.   For $\gamma \in \SL_2(\ZZ)$, the pairs $(z, H\cdot I)$ and $(\gamma^t(z), H\cdot \gamma^{-1})$ lies in the same class of $ (\mathfrak{H} \times (H\backslash \SL_2(\ZZ/N\ZZ) ))/\SL_2(\ZZ)$; from this one readily deduced that $\varphi$ is well-defined and injective.  It is straightforward to check that $\varphi$ is an isomorphism of Riemann surfaces.
\end{proof}

\subsection{$\FF_p$-points} \label{SS:Fp points}

Fix a prime $p\nmid 6N$ and an algebraic closure $\FFbar_p$ of $\FF_p$.  The Galois group $\Gal(\FFbar_p/\FF_p)$ is topologically generated by the automorphism $\Frob_p \colon x\mapsto x^p$.     In this section, we will describe how to compute $|X_G(\FF_p)|$.    \\

For an imaginary quadratic order $\OO$ of discriminant $D$, the $j$-invariant of the complex elliptic curve $\CC/\OO$ is an algebraic integer; its minimal polynomial $P_D(x) \in \ZZ[x]$ is the \defi{Hilbert class polynomial} of $\OO$.   For an integer $D<0$ which is not the discriminant of a quadratic order, we set $P_D(x)=1$.

Fix an elliptic curve $E$ over $\FF_p$ with $j_E \notin \{0,1728\}$.    Let $a_E$ be the integer $p+1-|E(\FF_p)|$.  Set $\Delta_E:=a_E^2-4p$; we have $\Delta_E \neq 0$ by the Hasse inequality.    Let $b_E$ be the largest integer $b\geq 1$ such that $b^2|\Delta_E$ and $P_{\Delta_E/b^2}(j_E) = 0$; this is well-defined since we will always have $P_{\Delta_E}(j_E)=0$.  Define the matrix
\[
\Phi_E := \begin{pmatrix}
     (a_E  - \Delta_E/b_E)/2 &  \Delta_E/b_E \cdot (1-\Delta_E/b_E^2)/4   \\
      b_E &  (a_E + \Delta_E/b_E)/2
\end{pmatrix};
\]
it has integer entries since $\Delta_E/b_E^2$ is an integer congruent to $0$ or $1$ modulo $4$ (it is the discriminant of a quadratic order) and $\Delta_E\equiv a_E \pmod{2}$.  One can check that $\Phi_E$ has trace $a_E$ and determinant $p$.  In practice, $\Phi_E$ is straightforward to compute; there are many good algorithms to compute $a_E$ and $P_D(x)$.

The following proposition shows that $\Phi_E$ describes $\rho_{E,N}(\Frob_p)$, and hence also $\rho_{E,N}$, up to conjugacy.   

\begin{prop} \label{P:Duke-Toth}
With notation as above, the reduction of $\Phi_E$ modulo $N$ is conjugate in $\GL_2(\ZZ/N\ZZ)$ to $\rho_{E,N}(\Frob_p)$.
\end{prop}
\begin{proof}
It suffices to prove the proposition when $N$ is a prime power.   For $N$ a prime power, it is then a consequence of Theorem~2 in \cite{Tommaso}.
\end{proof}

We now explain how to compute $|X_G(\FF_p)|$.  We can compute $|X_G^\infty(\FF_p)|$  using Lemma~\ref{L:cusp bijection} (with $k=\FF_p$, the subgroup $\chi_N(\Gal_{\FF_p})$ of $(\ZZ/N\ZZ)^\times$ is generated by $p$ modulo $N$).  So we need only describe how to compute  $|Y_G(\FF_p)|$;   it thus suffices to compute each term in the sum
\[
|Y_G(\FF_p)| = \sum_{j \in \FF_p} |\{P \in Y_G(\FF_p) : \pi_G(P)=j \}|.
\]
Take any $j \in \FF_p$ and fix an elliptic curve $E$ over $\FF_p$ with $j_E=j$.     \\

First suppose that $j\notin \{0,1728\}$.   We have $\Aut(E_{\FFbar_p})=\{\pm I\}$ and hence each automorphism acts on $E[N]$ by $I$ or $-I$.     Let $M$ be the group of isomorphisms $E[N] \xrightarrow{\sim} (\ZZ/N\ZZ)^2$.   Since $-I \in G$, we have $G \backslash M /\Aut(E_{\FFbar_p}) = G \backslash M$.    Lemma~\ref{L:fiber cardinality} implies that $ |\{P \in Y_G(\FF_p) : \pi_G(P)=j \}|$ is equal to cardinality of the subset of $G\backslash M$ fixed by the action of $\Frob_p$.    By Proposition~\ref{P:Duke-Toth} and choosing an appropriate basis of $E[N]$, we deduce that $|\{P \in Y_G(\FF_p) : \pi_G(P)=j \}|$ is equal to the cardinality of the subset of $G\backslash \GL_2(\ZZ/N\ZZ)$ fixed by right multiplication by $\Phi_E$.  In particular, note that we can compute $|\{P \in Y_G(\FF_p) : \pi_G(P)=j \}|$ without having to compute $E[N]$.\\

Now suppose that $j\in \{0,1728\}$ and recall that $p\nmid 6$.   When $j=0$, we take $E/\FF_p$ to be the curve defined by $y^2=x^3-1$; the group $\Aut(E_{\FFbar_p})$ is cyclic of order $6$ and generated by $(x,y)\mapsto (\zeta x, -y)$, where $\zeta \in \FFbar_p$ is a cube root of unity.   When $j =1728$, we take $E/\FF_p$ to be the curve defined by $y^2=x^3-x$; the group $\Aut(E_{\FFbar_p})$ is cyclic of order $6$ and generated by $(x,y)\mapsto (-x, \zeta y)$, where $\zeta\in \FFbar_p$ is a fourth root of unity. 

One can compute an explicit basis of $E[N]$.   With respect to this basis, the action of $\Aut(E_{\FFbar_p})$ on $E[N]$ corresponds to a subgroup $\calA$ of $\GL_2(\ZZ/N\ZZ)$ and the action of $\Frob_p$ on $E[N]$ corresponds to a matrix $\Phi_{E,N} \in \GL_2(\ZZ/N\ZZ)$.     Lemma~\ref{L:fiber cardinality} implies that $ |\{P \in Y_G(\FF_p) : \pi_G(P)=j \}|$ equals the number of elements in $G\backslash \GL_2(\ZZ/N\ZZ) / \calA$ that are fixed by right multiplication by $\Phi_{E,N}$.

\section{Preliminary work}  \label{P:theoretic work}

Take any congruence subgroup $\Gamma$ of $\SL_2(\ZZ)$ and denote its level by $N_0$.   Let $\pm \Gamma$ be the congruence subgroup generated by $\Gamma$ and $-I$.   Let $N$ be the integer $N_0$, $4N_0$ or $2N_0$ when $v_2(N_0)$ is $0$, $1$ or at least $2$, respectively.   

\begin{defn} \label{D:I(Gamma)}
We define $\scrI(\Gamma)$ to be the set of integers
\[
[\SL_2(\ZZ_N): G']\cdot {2}/{\gcd(2,N)},
\]
where $G$ varies over the open subgroups of $\GL_2(\ZZ_N)$ that are the inverse image by the reduction map $\GL_2(\ZZ_N) \to \GL_2(\ZZ/N\ZZ)$ of a subgroup $G(N)\subseteq \GL_2(\ZZ/N\ZZ)$ which satisfies the following conditions:
\begin{alphenum}
\item \label{I:a}
$G(N) \cap \SL_2(\ZZ/N\ZZ)$ is equal to $\pm \Gamma$ modulo $N$,
\item \label{I:b}
$G(N) \supseteq (\ZZ/N\ZZ)^\times\cdot I$,
\item \label{I:c}
$\det(G(N))=(\ZZ/N\ZZ)^\times$,
\item \label{I:d}
$G(N)$ contains a matrix that is conjugate to $\left(\begin{smallmatrix}1 & 0 \\0 & -1\end{smallmatrix}\right)$ or $\left(\begin{smallmatrix}1 & 1 \\0 & -1\end{smallmatrix}\right) $ in $\GL_2(\ZZ/N\ZZ)$,
\item \label{I:e}
the set $X_{G(N)}(\QQ)$ is infinite.
\end{alphenum}
\end{defn}

The set $\scrI(\Gamma)$ is finite since there are only finitely many possible $G(N)$ for a fixed $N$.   In the special case $N=1$, we view $\GL_2(\ZZ_N)$ and $\SL_2(\ZZ_N)$ as trivial groups and hence we find that $\scrI(\SL_2(\ZZ))=\{2\}$.   Define the set of integers 
\[
\scrI := \bigcup_{\Gamma} \scrI(\Gamma),
\]
where the union is over the congruence subgroups of $\SL_2(\ZZ)$ that have genus $0$ or $1$.  The set $\scrI$ is finite since there are only finitely many congruence subgroups of genus $0$ or $1$, see~\cite{MR2016709}.

 The goal of this section is to prove the following theorem.

\begin{thm} \label{T:main ineffective scrI}
Fix an integer $c$.    There is a finite set $J$, depending only on $c$, such that if $E/\QQ$ is an elliptic curve with $j_E \notin J$ and $\rho_{E,\ell}$ surjective for all primes $\ell >c$, then  $[\GL_2(\Zhat) : \rho_E(\Gal_\QQ)]$ is an element of $\scrI$.
\end{thm}

In \S\ref{S:index}, we will compute $\scrI$ and show that it is equal to the set $\calI$ from \S\ref{S:intro}; this will prove Theorem~\ref{T:main ineffective}.

\subsection{The congruence subgroup $\Gamma_E$} \label{SS:indices scrI}

Fix a non-CM elliptic curve $E$ over $\QQ$.    Define the subgroup 
\[
G:=\Zhat^\times\cdot \rho_E(\Gal_\QQ)
\]
of $\GL_2(\Zhat)$.  For each positive integer $n$, let $G_n$ be the image of $G$ under the projection map $\GL_2(\Zhat)\to \GL_2(\ZZ_n)$.

By Serre's theorem, $G$ is an open subgroup of $\GL_2(\Zhat)$.      We have an equality $G'=\rho_E(\Gal_\QQ)'$ of commutator subgroups and hence
\begin{equation} \label{E:index to G}
[\GL_2(\Zhat):\rho_E(\Gal_\QQ)] = [\SL_2(\Zhat): G']
\end{equation}
by Proposition~\ref{P:EC equality}.  There is no harm in working with the larger group $G$ since we are only concerned about the index $[\GL_2(\Zhat):\rho_E(\Gal_\QQ)]$.    \\

 Let $m$ be the product of the primes $\ell$ for which $\ell\leq 5$ or for which $\rho_{E,\ell}$ is not surjective.  The group $G_{m}  \cap \SL_2(\ZZ_{m})$ is open in $\SL_2(\ZZ_{m})$.    Let $N_0\geq 1$ be the smallest positive integer dividing some power of $m$ for which
 \begin{align} \label{E:N definition}
 G_m \cap \SL_2(\ZZ_m) \supseteq \{A\in \SL_2(\ZZ_m): A\equiv I \pmod{N_0}\}.
 \end{align}
Let $N$ be the integer $N_0$, $4N_0$ or $2N_0$ when $v_2(N_0)$ is $0$, $1$ or at least $2$, respectively.   \\

Define $\Gamma_E := \Gamma_{G(N)}$; it is the congruence subgroup consisting of matrices in $\SL_2(\ZZ)$ whose image modulo $N$ lies in $G(N)$.  Note that the congruence subgroup $\Gamma_E$ has level $N_0$ and contains $-I$.  

\begin{prop} \label{P:scrI relevance}
The subgroup $G(N)$ of $\GL_2(\ZZ/N\ZZ)$ satisfies conditions (\ref{I:a}), (\ref{I:b}), (\ref{I:c}) and (\ref{I:d}) of Definition~\ref{D:I(Gamma)} with $\Gamma=\Gamma_E$.
\end{prop}
\begin{proof}
Our congruence subgroup $\Gamma_E$ contains $-I$ and was chosen so that $\Gamma_E$ modulo $N$ equals $G(N)\cap \SL_2(\ZZ/N\ZZ)$.  We have $G \supseteq \Zhat^\times \cdot I$, so $G(N) \supseteq (\ZZ/N\ZZ)^\times \cdot I$.    We have $\det(\rho_E(\Gal_\QQ))=\Zhat^\times$, so $\det(G(N))=(\ZZ/N\ZZ)^\times$.  

It remains to show that condition (\ref{I:d}) holds.    Since $E/\QQ$ is non-CM and $\rho_{E,N}(\Gal_\QQ)$ is a subgroup of $G(N)$, we have $Y_{G(N)}(\QQ)\neq \emptyset$ by Proposition~\ref{P:key}.    In particular, $Y_{G(N)}(\RR) \neq \emptyset$.   Proposition~\ref{P:R-points for YG} implies that $G$ contains an element that is conjugate in $\GL_2(\ZZ/N\ZZ)$ to $\left(\begin{smallmatrix}1 & 0 \\0 & -1\end{smallmatrix}\right)$ or $\left(\begin{smallmatrix}1 & 1 \\0 & -1\end{smallmatrix}\right)$.  
\end{proof}

The following lemma shows that $G_N$ is determined by $G(N)$.

 \begin{lemma} \label{L:image mod tildeN}
 The group $G_N$ is the inverse image of $G(N)$ under the reduction modulo $N$ map $\GL_2(\ZZ_N) \to \GL_2(\ZZ/N\ZZ)$.
 \end{lemma}
\begin{proof}
Take any $A\in \GL_2(\ZZ_N)$ satisfying $A\equiv I \pmod{N}$; we need only verify that $A$ is an element of $G_N$.   Our integer $N$ has the property that $(1+N_0\ZZ_N)^2 = 1 + N \ZZ_N$.  Since $\det(A)\equiv 1 \pmod{N}$, we have $\det(A)=\lambda^2$ for some $\lambda\in 1+ N_0\ZZ_N$.  Define $B:=\lambda^{-1}A$; it is an element of $\SL_2(\ZZ_N)$ that is congruent to $I$ modulo $N_0$.     Using (\ref{E:N definition}), we deduce that $B$ is an element of $G_N$.   From the definition of $G$, it is clear that $G_N$ contains the scalar matrix $\lambda I$.    Therefore, $A= \lambda I \cdot B$ is an element of $G_N$.
\end{proof} 

The following group theoretical lemma will be proved in \S\ref{SS:proof technical group}. 

\begin{lemma}   \label{L:technical group section}
We have 
\[
[\SL_2(\ZZ_N): G'] = [\SL_2(\ZZ_m): G_m']  = [\SL_2(\ZZ_N): G_N'] \cdot 2/\gcd(2,N).
\]
Moreover, $G' = G_m' \times \prod_{\ell \nmid m} \SL_2(\ZZ_\ell)$.
\end{lemma}

The following lemma motivates our definition of $\scrI$.

\begin{lemma} \label{L:contained in scrI}
If $X_{G(N)}(\QQ)$ is infinite, then $[\GL_2(\Zhat): \rho_E(\Gal_\QQ)]$ is an element of $\scrI$.
\end{lemma}
\begin{proof}
By Lemma~\ref{L:technical group section} and (\ref{E:index to G}), we have $[\GL_2(\Zhat):\rho_E(\Gal_\QQ)]=[\SL_2(\ZZ_N): G_N'] \cdot 2/\gcd(2,N)$.

The group $G(N)$ satisfies conditions (\ref{I:a}), (\ref{I:b}), (\ref{I:c}) and (\ref{I:d}) of Definition~\ref{D:I(Gamma)} with $\Gamma=\Gamma_E$ by Lemma~\ref{L:image mod tildeN}.    The group $G(N)$ satisfies (\ref{I:e}) by assumption.     Using Lemma~\ref{L:image mod tildeN}, we deduce that $[\SL_2(\ZZ_N): G_N'] \cdot 2/\gcd(2,N)$ is an element of $\scrI(\Gamma_E)$.

To complete the proof of the lemma, we need to show that $\Gamma_E$ has genus $0$ or $1$ since then $\scrI(\Gamma_E)\subseteq \scrI$.   The genus of $\Gamma_E$ is equal to the genus of $X_{G(N)}$ by Proposition~\ref{P:isomorphic Riemann surfaces}.    Since $X_{G(N)}$ has infinitely many rational point, it must have genus $0$ or $1$ by Faltings' theorem.
\end{proof}

\subsection{Exceptional rational points on modular curves} \label{SS:exceptional points}

Let $\calS$ be the set of pairs $(N,G)$ with $N\geq 1$ an integer not divisible by any prime $\ell >13$ and with $G$ a subgroup of $\GL_2(\ZZ/N\ZZ)$ satisfying the following conditions:
\begin{itemize}
\item
$\det(G)=(\ZZ/N\ZZ)^\times$ and $-I \in G$,
\item
$X_G$ has genus at least $2$ or $X_G(\QQ)$ is finite.
\end{itemize}
Define the set
 \[
 \calJ :=  \bigcup_{(N,G) \in \calS} \pi_G(Y_G(\QQ)).
 \]
We will prove that $\calJ$ is finite.  We will need the following lemma.

\begin{lemma} \label{L:Frattini}
Fix an integer $m\geq 2$.  An open subgroup $H$ of $\GL_2(\ZZ_m)$ has only a finite number of closed maximal subgroups and they are all open.
\end{lemma}
\begin{proof}
The lemma follows from the proposition in \cite{MR1757192}*{\S10.6} which gives a condition for the \emph{Frattini subgroup} of $H$ to be open;  note that $H$ contains a normal subgroup of the form $I+m^e M_2(\ZZ_m)$ for some $e\geq 1$ and that $I+m^e M_2(\ZZ_m)$ is the product of pro-$\ell$ groups with $\ell |m$. 
\end{proof}

\begin{prop} \label{P:finiteness}
The set $\calJ$ is finite.
\end{prop}
\begin{proof}
Fix pairs $(N,G), (N',G') \in \calS$ such that $N$ is a divisor of $N'$ and such that reduction modulo $N$ gives a well-defined map $G'\to G$.   This gives rise to a morphism $\varphi \colon Y_{G'} \to Y_G$ of curves over $\QQ$ such that $\pi_G \circ \varphi = \pi_{G'}$.   In particular, $\pi_{G'}(Y_{G'}(\QQ)) \subseteq \pi_G(Y_G(\QQ))$.    Therefore,
\[
\calJ = \bigcup_{(N,G) \in \calS'}  \pi_{G}(Y_G(\QQ)),
\]
where $\calS'$ is the set of pairs $(N,G) \in \calS$ for which there is no pair $(N',G') \in \calS - \{(N,G)\}$ with  $N'$ a divisor of $N$ so that the reduction modulo $N'$ defines a map $G \to G'$.    For each pair $(N,G) \in \calS'$, the set $Y_G(\QQ)$, and hence also $\pi_G(Y_G(\QQ))$, is finite.  The finiteness is immediate from the definition of $\calS$ when $Y_G$ has genus $0$ or $1$.  If $Y_G$ has genus at least $2$, then $Y_G(\QQ)$  is finite by Faltings' theorem.   So to prove that $\calJ$ is finite, it suffices to show that $\calS'$ is finite.\\

Let $m$ be the product of primes $\ell\leq 13$. For each pair $(N,G) \in \calS'$, let $\tilde{G}$ be the open subgroup of $\GL_2(\ZZ_m)$ that is the inverse image of $G$ under the reduction map $\GL_2(\ZZ_m) \to \GL_2(\ZZ/N\ZZ)$.   Note that we can recover the pair $(N,G)$ from $\tilde{G}$; $N \geq 1$ is the smallest integer (not divisible by primes $\ell>13$) such that $\tilde{G}$ contains $\{A\in \GL_2(\ZZ_m) : A \equiv I \pmod{N}\}$ and $G$ is the image of $\tilde{G}$ in $\GL_2(\ZZ/N\ZZ)$.     Define the set 
\[
\calG := \{ \tilde{G} : (N,G) \in \calS' \}.
\]
We have $|\calG| = |\calS'|$, so it suffices to show that the set $\calG$ is finite.

Suppose that $\calG$ is infinite.   We now recursively define a sequence $\{M_i\}_{i\geq 0}$ of open subgroups of $\GL_2(\ZZ_m)$ such that
\begin{align} \label{E:Mi chain}
M_0 \supsetneq M_1 \supsetneq M_2 \supsetneq M_3 \supsetneq \ldots
\end{align}
and such that each $M_i$ has infinitely many subgroups in $\calG$.   Set $M_0 := \GL_2(\ZZ_m)$.   Take an $i\geq 0$ for which $M_i$ has been defined and has infinitely many subgroups in $\calG$.   Since $M_i$ has only finite many open maximal subgroups by Lemma~\ref{L:Frattini}, one of the them contains infinitely many subgroups in $\calG$; denote such a maximal subgroup by $M_{i+1}$.

Take any $i\geq 0$.   Since there are elements of $\calG$ that are proper subgroups of $M_i$, we deduce that $M_i \supsetneq \tilde{G}$ for some pair $(N,G) \in \calS'$.   The group $G = \tilde{G}(N)$ is thus a proper subgroup of $M_i(N) \subseteq \GL_2(\ZZ/N\ZZ)$.   We have $\det(M_i(N))=(\ZZ/N\ZZ)^\times$ and $-I \in M_i(N)$ since $G$ has these properties.       We have $(N,M_i(N)) \notin \calS$ since otherwise $(N,G)$ would not be an element of $\calS'$.   Therefore, the modular curve $X_{M_i(N)}$ has genus $0$ or $1$.    By Proposition~\ref{P:isomorphic Riemann surfaces}, the congruence subgroup $\Gamma_i := \Gamma_{M_i(N)}$ (which consists of $A\in \SL_2(\ZZ)$ with $A$ modulo $N$ in $M_i(N)$) has genus $0$ or $1$.   We have 
\[
[\SL_2(\ZZ):\Gamma_i] = [\SL_2(\ZZ/N\ZZ): M_i(N) \cap \SL_2(\ZZ/N\ZZ) ] = [\GL_2(\ZZ/N\ZZ) : M_i(N)] = [\GL_2(\ZZ_m): M_i],
\]
so $[\SL_2(\ZZ):\Gamma_i] \to \infty$ as $i\to \infty$ by the proper inclusions (\ref{E:Mi chain}).    In particular, there are infinitely many congruence subgroup of genus $0$ or $1$.  However, there are only finitely many congruence subgroups of $\SL_2(\ZZ)$ of genus $0$ and $1$; moreover, the level of such congruence subgroups is at most $52$ by \cite{MR2016709}.  This contradiction implies that $\calG$, and hence $\calS'$, is finite.  
\end{proof}

For each prime $\ell$, let $\calJ_\ell$ be the set of $j$-invariants of elliptic curves $E/\QQ$ for which $\rho_{E,\ell}$ is not surjective.

\begin{prop} \label{P:finiteness ell}
The set $\calJ_\ell$ is finite for all primes $\ell >13$.
\end{prop}
\begin{proof}
Fix a prime $\ell > 13$.    By Proposition~\ref{P:key}, it suffices to show that $X_G(\QQ)$ is finite for each of the maximal subgroups $G$ of $\GL_2(\ZZ/\ell\ZZ)$ that satisfy $\det(G)=(\ZZ/\ell\ZZ)^\times$.   Fix such a group $G$ and let $\Gamma=\Gamma_G$ be the congruence subgroup consisting of $A\in \SL_2(\ZZ)$ for which $A$ modulo $N$ lies in $G$.   The curve $X_G$ has the same genus as $\Gamma$ by Proposition~\ref{P:isomorphic Riemann surfaces}.     If $\Gamma$ has genus at least $2$, then $X_G(\QQ)$ is finite by Faltings' theorem.

We may thus suppose that $\Gamma$ has genus $0$ or $1$.   From the description of congruence subgroups of genus $0$ and $1$ in \cite{MR2016709}, we find that $\ell \in \{17,19\}$ and that $\Gamma$ modulo $\ell$ contains an element of order $\ell$.   Therefore, after replacing $G$ by a  conjugate in $\GL_2(\ZZ/\ell\ZZ)$, we may assume that $G$ is the subgroup of upper-triangular matrices.    So we are left to consider the modular curve $X_0(\ell) := X_G$ with $\ell \in \{17,19\}$.   The curve $X_0(\ell)$, with $\ell \in \{17,19\}$, indeed has finitely many points (it has a rational cusp, so it is an elliptic curve of conductor $\ell \in \{17,19\}$; all such elliptic curves have rank $0$).
\end{proof}  
\subsection{Proof of Theorem~\ref{T:main ineffective scrI}}
Let $\calJ$ and $\calJ_\ell$ (with $\ell>13$) be the sets from \S\ref{SS:exceptional points}.    Define the set 
\[
J:= \calJ \cup \bigcup_{13 < \ell  \leq c} \calJ_\ell;
\]
it is finite by Propositions~\ref{P:finiteness} and \ref{P:finiteness ell}.

Take any elliptic curve $E/\QQ$ with $j_E \notin J$ for which $\rho_{E,\ell}$ is surjective for all $\ell>c$.     Since $j_E \notin J_\ell$ for $13<\ell\leq c$, the representation $\rho_{E,\ell}$ is surjective for all $\ell > 13$.

  Let $\Gamma_E$ be the congruence subgroup from \S\ref{SS:indices scrI}; denote its level by $N_0$ and define $N$ as in the beginning of the section.    Let $G(N)$ be the subgroup of $\GL_2(\ZZ/N\ZZ)$ from \S\ref{SS:indices scrI} associated to $E/\QQ$.

\begin{lemma} \label{L:condition e holds}
The set $X_{G(N)}(\QQ)$ is infinite.
\end{lemma}
\begin{proof}
Take $\calS$ as in \S\ref{SS:exceptional points}.  
The integer $N$ is not divisible by any prime $\ell>13$ since $\rho_{E,\ell}$ is surjective for all $\ell>13$.    If $(N,G(N)) \in \calS$, then $j_E \in \pi_{G(N)}(Y_{G(N)}(\QQ)) \subseteq \calJ \subseteq J$.    Since $j_E \notin J$ by assumption, we have $(N,G(N)) \notin \calS$.   We have $\det(G(N))=(\ZZ/N\ZZ)^\times$ and $-I \in G(N)$, so $(N,G(N)) \notin \calS$ implies that $X_{G(N)}$ has genus $0$ or $1$, and that $X_{G(N)}(\QQ)$ is infinite.   
\end{proof}

Lemmas~\ref{L:contained in scrI} and \ref{L:condition e holds} together imply that $[\GL_2(\Zhat): \rho_E(\Gal_\QQ)]$ is an element of $\scrI$.

\subsection{Proof of Lemma~\ref{L:technical group section}} \label{SS:proof technical group}

Let $d$ be the product of primes that divide $m$ but not $N$; it divides $2\cdot 3\cdot 5$.  Since $G_m \cap \SL_2(\ZZ_m)$ contains $\{A\in \SL_2(\ZZ_m): A\equiv I \pmod{N_0}\}$, we have
\[
G_m \cap \SL_2(\ZZ_m) = W \times \SL_2(\ZZ_d).
\]
for a subgroup $W$ of $\SL_2(\ZZ_N)$ containing $\{A\in \SL_2(\ZZ_N): A\equiv I \pmod{N_0}\}$.   Since $G_m\cap\SL_2(\ZZ_m)$ is a normal subgroup of $G_m$, the group $W$ is normal in $G_N$.    We have $G_d=\GL_2(\ZZ_d)$, since $G_d\supseteq \SL_2(\ZZ_d)$ and $\det(G_d)=\ZZ_d^\times$ (note that $\det(\rho_{E}(\Gal_\QQ))=\Zhat^\times$).  

 Now consider the quotient map 
 \[
 \varphi\colon G_N \times G_d \to G_N/W \times G_d/\SL_2(\ZZ_d).
 \]  
 We can view $G_m$ as an open subgroup of $G_N\times G_d$; it projects surjectively on both of the factors.   The group $G_m$ contains $W\times \SL_2(\ZZ_d)$, so there is an open subgroup $Y$ of $G_N/W \times G_d/\SL_2(\ZZ_d)$ for which $G_m = \varphi^{-1}(Y)$.  
 
 Take any matrices $B_1, B_2 \in G_d=\GL_2(\ZZ_d)$ with $\det(B_1)=\det(B_2)$; equivalently, with the same image in $G_d/\SL_2(\ZZ_d)$.    There is a matrix $A\in G_N$ such that $(A,B_1)\in G_m$ and hence also $(A,B_2)\in G_m$ since $\varphi(A,B_1)=\varphi(A,B_2)$.    Therefore, the commutator subgroup $G_{m}'$ contains the element
\[
(A,B_1)\cdot (A,B_2)\cdot (A,B_1)^{-1} \cdot (A,B_2)^{-1} = (I, B_1 B_2 B_1^{-1} B_2^{-1}).
\]
By Lemma~\ref{L:SL2 derived group theory}(iv) below, the group $\GL_2(\ZZ_d)'$ is topologically generated by the set
\[
\big\{ B_1 B_2 B_1^{-1} B_2^{-1} : B_1,B_2\in \GL_2(\ZZ_d),\, \det(B_1)=\det(B_2) \big\},
\]
and hence $G_m' \supseteq \{I\} \times \GL_2(\ZZ_d)'$.    We have an inclusion $G_m' \subseteq G_N' \times G_d'=G_N' \times \GL_2(\ZZ_d)'$ and the projections of $G_m'$ onto the first and second factors are both surjective; since $G_m \supseteq \{I\} \times \GL_2(\ZZ_d)'$ we find that 
\begin{align} \label{E:Gm commutator}
G_m'=G_N' \times \GL_2(\ZZ_d)'.
\end{align} 

\begin{lemma} \label{L:SL2 derived group theory}
\begin{romanenum}  
\item \label{L:SL2 derived group theory a}
For $\ell\geq 5$, we have $\SL_2(\ZZ_\ell)' = \SL_2(\ZZ_\ell)$.
\item \label{L:SL2 derived group theory b} 
For $\ell=2$ or $3$, let $b=4$ or $3$, respectively.  Then reduction modulo $b$ induces an isomorphism 
\[
\SL_2(\ZZ_\ell)/\SL_2(\ZZ_\ell)' \xrightarrow{\sim} \SL_2(\ZZ/b\ZZ)/\SL_2(\ZZ/b\ZZ)'
\]
of cyclic groups of order $b$.
\item  \label{L:SL2 derived group theory c}
We have $\GL_2(\ZZ_3)'=\SL_2(\ZZ_3)$ and $[\SL_2(\ZZ_2):\GL_2(\ZZ_2)']=2$.
\item \label{L:SL2 derived group theory d}
For each positive integer $d$, the group $\GL_2(\ZZ_d)'$ is topologically generated by the set 
\[
\{ ABA^{-1}B^{-1} : A, B \in \GL_2(\ZZ_d), \, \det(A)=\det(B) \}.
\]
\end{romanenum}
\end{lemma}
\begin{proof}
For part (\ref{L:SL2 derived group theory a}) and (\ref{L:SL2 derived group theory b}), see \cite{Zywina-Maximal}*{Lemma~A.1}.  To verify (\ref{L:SL2 derived group theory c}),  it suffices by (\ref{L:SL2 derived group theory b}) to show that $\GL_2(\ZZ/3\ZZ)' = \SL_2(\ZZ/3\ZZ)$ and $[\SL_2(\ZZ/4\ZZ): \GL_2(\ZZ/4\ZZ)'] = 2$; this is an easy computation.

Finally consider (\ref{L:SL2 derived group theory d}).  Without loss of generality, we may assume that $d$ is a prime, say $\ell$.   The topological group generated by the set $\mathcal{C}=\{ ABA^{-1}B^{-1} : A, B \in \GL_2(\ZZ_\ell), \, \det(A)=\det(B) \}$ contains $\SL_2(\ZZ_\ell)'$, so it suffices to show that the image of $\mathcal{C}$ generates $\GL_2(\ZZ_\ell)'/\SL_2(\ZZ_\ell)'$.   If $\ell\geq 5$, this is trivial since $\GL_2(\ZZ_\ell)'$ and $\SL_2(\ZZ_\ell)'$ both equal $\SL_2(\ZZ_\ell)$ by (\ref{L:SL2 derived group theory a}).   For $\ell=2$ or $3$, it suffices by part (\ref{L:SL2 derived group theory b}) to show that $\GL_2(\ZZ/b\ZZ)'$ is generated by $ABA^{-1}B^{-1}$ with matrices $A,B\in \GL_2(\ZZ/b\ZZ)$ having the same determinant; this again is an easy calculation.
\end{proof}

Before computing $G'$, we first state Goursat's lemma; we will give a more general version than needed so that it can be cited in future work.

\begin{lemma}[Goursat's Lemma] \label{L:Goursat}  
Let $B_1,\ldots, B_n$ be profinite groups.  Assume that for distinct $1\leq i, j \leq n$, the groups $B_i$ and $B_j$ have no finite simple groups as common quotients.  Suppose that $H$ is a closed subgroup of $\prod_{i=1}^n B_i$ that satisfies $p_j(H)=B_j$ for all $j$ where $p_j\colon \prod_{i=1}^n B_i \to B_j$ is the projection map.    Then $H=\prod_{i=1}^n B_i$.
\end{lemma}
\begin{proof}
We proceed by induction on $n$.   The case $n=1$ is trivial, so assume that $n=2$.  The kernel of $p_1|_H$ is a closed subgroup of $H$ of the form $\{I\}\times N_2$, and similarly the kernel of $p_2|_H$ is of the form $N_1\times\{I\}$.   The group $N=N_1\times N_2$ is a closed normal subgroup of $H$.   Since $p_1|_H$ is surjective, we find that $N_1=p_1(N)$ is a closed normal subgroup of $B_1$; this gives an isomorphism $H/N\cong B_1/N_1$ of profinite groups.  Similarly, we have $H/N\cong B_2/N_2$ and thus $B_1/N_1$ and $B_2/N_2$ are isomorphism.  Since we have assumed that $B_1$ and $B_2$ have no common finite simple quotients, we deduce that $B_1=N_1$ and $B_2=N_2$.  This proves the $n=2$ case since $H$ contains $N_1\times N_2=B_1\times B_2$.   

Now fix an $n\geq 3$ and assume that the $n-1$ case of the lemma has been proved.   Then the image $\tilde{H}$ of $H$ in $C:=\prod_{i=1}^{n-1}B_i$ is a closed subgroup such that the projection $\tilde{H}\to B_i$ is surjective for all $1\leq i \leq n-1$.  By our inductive hypothesis, we have $\tilde{H}=C$.   So $H$ is a closed subgroup of $C\times B_n$ and the projections $H\to C$ and $H\to B_n$ are surjective.  By the $n=2$ case, it suffices to show any finite simple quotient of $C$ is not a quotient of $B_n$.  Take any open normal subgroup $U$ of $C$ such that $C/U$ is a finite simple group.   There is an integer $1\leq j \leq n-1$ for which the projection $U\to B_j$ is not surjective (if not, then we could use our inductive hypothesis to show that $U=C$).   For simplicity, suppose $j=1$; then $U$ is of the form $N_1 \times B_2 \times \cdots\times B_{n-1}$ where $N_1$ is an open normal subgroup of $B_1$.   Since $C/U\cong B_1/N_1$, we deduce from the hypothesis on the $B_i$ that $C/U$ is not a quotient of $B_n$.
\end{proof}

We claim that $G_\ell' = \SL_2(\ZZ_\ell)$ for every prime $\ell \nmid m$.    We have the easy inclusions $G_\ell '  \subseteq \GL_2(\ZZ_\ell)' \subseteq \SL_2(\ZZ_\ell)$.     By \cite{MR1043865}*{IV~Lemma~3} and $\ell >5$ (since $\ell\nmid m$), we have $G_\ell'=\SL_2(\ZZ_\ell)$ if and only if the image of $G_\ell'$ in $\SL_2(\ZZ/\ell\ZZ)$ is $\SL_2(\ZZ/\ell\ZZ)$.    It thus suffices to show that $\rho_{E,\ell}(\Gal_\QQ)' = \SL_2(\ZZ/\ell\ZZ)$.     Since $\ell \nmid m$, we have $\rho_{E,\ell}(\Gal_\QQ) = \GL_2(\ZZ/\ell\ZZ)$ and hence $\rho_{E,\ell}(\Gal_\QQ)' = \SL_2(\ZZ/\ell\ZZ)$ by Lemma~\ref{L:SL2 derived group theory}(\ref{L:SL2 derived group theory a}); this proves our claim.\\

We can view $G'$ as a subgroup of $G_m' \times \prod_{\ell\nmid m} \SL_2(\ZZ_\ell)$.    The projection of $G'$ to the the factors $G_m'$ and $\SL_2(\ZZ_\ell)=G_\ell'$ with $\ell\nmid m$ are all surjective.

Fix a prime $\ell\geq 5$.  The simple group $\PSL_2(\FF_\ell)$ is a quotient of $\SL_2(\ZZ_\ell)$.  Since $\ell$-groups are solvable and $\SL_2(\ZZ_\ell)'=\SL_2(\ZZ_\ell)$ by Lemma~\ref{L:SL2 derived group theory}(\ref{L:SL2 derived group theory a}), we find that $\PSL_2(\FF_\ell)$ is the only simple group that is a quotient of $\SL_2(\ZZ_\ell)$.   Note that the groups $\PSL_2(\FF_\ell)$ are non-isomorphic for different $\ell$; in fact, they have different cardinalities.   

Take any prime $\ell\nmid m$, and hence $\ell>5$.  We claim that the simple group $\PSL_2(\FF_\ell)$ is not isomorphic to a quotient of $G_m'$.  Indeed, \emph{any} closed subgroup $H$ of $\GL_2(\ZZ_m)$ has no quotients isomorphic to $\PSL_2(\FF_\ell)$ with $\ell>5$ and $\ell\nmid m$ (this follows from the calculation of the groups $\textrm{Occ}(\GL_2(\ZZ_\ell))$ in \cite{MR1484415}*{IV-25}).   We can now apply Goursat's lemma (Lemma~\ref{L:Goursat}) to deduce that 
\[
G'=G_m' \times \prod_{\ell\nmid m} \SL_2(\ZZ_\ell).
\]   
Therefore, $[\SL_2(\Zhat): G'] = [\SL_2(\ZZ_m) : G_m']$.     By (\ref{E:Gm commutator}), we have 
\[
[\SL_2(\ZZ_m) : G_m']=[\SL_2(\ZZ_N): G_N'] \cdot [\SL_2(\ZZ_d): \GL_2(\ZZ_d)']. 
\]
By Lemma~\ref{L:SL2 derived group theory}, $[\SL_2(\ZZ_d): \GL_2(\ZZ_d)']=\prod_{\ell |d} [\SL_2(\ZZ_\ell): \GL_2(\ZZ_\ell)']$ is equal to $1$ if $d$ is odd and $2$ if $d$ is even.   Since $N$ and $d$ have opposite parities, we conclude that $[\SL_2(\ZZ_m) : G_m']$ is equal to $[\SL_2(\ZZ_N): G_N']$ if $N$ is even and $[\SL_2(\ZZ_N): G_N']\cdot 2$ if $N$ is odd.    The lemma is now immediate.

\section{Index computations} \label{S:index}

In \S\ref{SS:main results}, we defined the set
\[
\calI= \left\{\begin{array}{c}2, 4, 6, 8, 10, 12, 16, 20, 24, 30, 32, 36, 40, 48, 54, 60, 72, 84, 96, 108, 112,120, 144, \\192, {220}, {240},  288, 336, {360},  384, {504}, 576, 768, 864, 1152, 1200, 1296, 1536 \end{array}\right\}.
\]
In \S\ref{P:theoretic work}, we defined the set of integers
\[
\scrI:= {\bigcup}_{\Gamma} \scrI(\Gamma),
\]
where $\Gamma$ runs over the congruence subgroups of $\SL_2(\ZZ)$ of genus $0$ or $1$.  The goal of this section is to outline the computations needed to verify the following.

\begin{prop} \label{P:calI computation}
We have $\scrI = \calI$.
\end{prop}

The computations in this section were performed with \texttt{Magma} \cite{Magma}; code for the computations can be found at
\begin{center}
\url{https://github.com/davidzywina/PossibleIndices}
\end{center}

Let $S_0$ and $S_1$ be sets of representatives of the congruence subgroups of $\SL_2(\ZZ)$ containing $-I$, up to conjugacy in $\GL_2(\ZZ)$, with genus $0$ and $1$, respectively.  Set $S:=S_0 \cup S_1$.   Since the set $\scrI(\Gamma)$ does not change if we replace $\Gamma$ by $\pm \Gamma$ or by a conjugate subgroup in $\GL_2(\ZZ)$, we have
\[
\scrI = {\bigcup}_{\Gamma \in S} \scrI(\Gamma).
\]

Cummin and Pauli \cite{MR2016709} have classified the congruence subgroups of $\PSL_2(\ZZ)$ with genus $0$ or $1$, up to conjugacy in $\PGL_2(\ZZ)$.   We thus have a classification of the congruence subgroups $\Gamma$ of $\SL_2(\ZZ)$, up to conjugacy in $\GL_2(\ZZ)$, of genus $0$ or $1$ that contain $-I$.       Moreover, they have made available an explicit list\footnote{See \url{http://www.uncg.edu/mat/faculty/pauli/congruence/congruence.html}} of such congruence subgroups; each congruence subgroup is given by a level $N$ and set of generators of its image in $\SL_2(\ZZ/N\ZZ)/\{\pm I\}$.  In our computations, we will let $S_0$ and $S_1$ consist of congruence subgroups from the explicit list of Cummin and Pauli.

\subsection{Computing indices}  \label{SS:how to compute}

Fix a congruence subgroup $\Gamma$ of $\SL_2(\ZZ)$ that contains $-I$ and has level $N_0$.  Let $N$ be the integer $N_0$, $4N_0$ or $2N_0$ when $v_2(N_0)$ is $0$, $1$ or at least $2$, respectively.  For simplicity, we will assume that $N>1$.

We first explain how we computed the subgroups $G(N)$ of $\GL_2(\ZZ/N\ZZ)$ that satisfy conditions (\ref{I:a}), (\ref{I:b}) and (\ref{I:c}) of Definition~\ref{D:I(Gamma)}.  Instead of directly looking for subgroups in $\GL_2(\ZZ/N\ZZ)$, we will search for certain abelian subgroups in a smaller group.

Let $H$ be the the image of $\pm \Gamma =\Gamma$ in $\SL_2(\ZZ/N\ZZ)$.  Define the subgroup $\widetilde{H} := (\ZZ/N\ZZ)^\times \cdot H$ of $\GL_2(\ZZ/N\ZZ)$.   We may assume that $H = \widetilde{H} \cap \SL_2(\ZZ/N\ZZ)$; otherwise, conditions (\ref{I:a}) and (\ref{I:b}) are incompatible.

Let $\calN$ be the normalizer of $\widetilde{H}$ (equivalently, of $H$) in $\GL_2(\ZZ/N\ZZ)$ and set $\calC:= \calN/\widetilde{H}$.   Since $\det(\widetilde{H})=((\ZZ/N\ZZ)^\times)^2$, the determinant induces a homomorphism \[
\det\colon \calC \to  (\ZZ/N\ZZ)^\times / ((\ZZ/N\ZZ)^\times )^2=:Q_N.
\]

\begin{lemma} \label{L:G(N) vs W}
The subgroups $G(N)$ of $\GL_2(\ZZ/N\ZZ)$ that satisfy conditions (\ref{I:a}), (\ref{I:b}) and (\ref{I:c}) of Definition~\ref{D:I(Gamma)} are precisely the groups obtained by taking the inverse image under $\calN \to \calC$ of the subgroups $W$ of $\calC$ for which the determinant induces an isomorphism $W\xrightarrow{\sim} Q_N$.
\end{lemma}
\begin{proof}
Let $B:=G(N)$ be a subgroup of $\GL_2(\ZZ/N\ZZ)$ that satisfies conditions (\ref{I:a}), (\ref{I:b}) and (\ref{I:c}).   The group $B$ contains $\widetilde{H}$ by (\ref{I:a}) and (\ref{I:b}).  For any matrix $A\in  B$ with $\det(A)$ a square, there is a scalar $\lambda \in (\ZZ/N\ZZ)^\times$ such that $\det(\lambda A)=1$.  Since $B \cap \SL_2(\ZZ/N\ZZ) = H$ by (\ref{I:a}), we deduce that $\widetilde{H}$ consists precisely of the element of $B$ with square determinant.      The determinant thus gives rise to an exact sequence 
\begin{align} \label{E:QN ses}
1\to \widetilde{H} \hookrightarrow B \xrightarrow{\det} Q_N \to 1.
\end{align}
Therefore, $\widetilde{H}$ is a normal subgroup of $B$, and hence $B \subseteq \calN$, and the determinant map induces an isomorphism $B/\widetilde{H} \xrightarrow{\sim} Q_N$.    Let $W$ be the image of the natural injection $B/\widetilde{H}\hookrightarrow \calN/\widetilde{H}=\calC$; it satisfies the conditions for $W$ in the statement of the lemma.

Now take any subgroup $W$ of $\calC$ for which the determinant gives an isomorphism $W\xrightarrow{\sim} Q_N$.  Let $B$ be the inverse image of $W$ under the map $\calN \to \calC$.   The short exact sequence (\ref{E:QN ses}) holds.   Therefore, $B \cap \SL_2(\ZZ/N\ZZ)$ is equal to $\widetilde{H} \cap \SL_2(\ZZ/N\ZZ) = H$.    We have $B\supseteq (\ZZ/N\ZZ)^\times\cdot I$ since $B\supseteq \widetilde{H}$.    So $\det(B) \supseteq ((\ZZ/N\ZZ)^\times)^2$;  with $\det(B/\widetilde{H}) = Q_N$, this implies that $\det(B) = (\ZZ/N\ZZ)^\times$.   We have verified that $G(N):=B$ satisfies conditions (\ref{I:a}), (\ref{I:b}) and (\ref{I:c}).
\end{proof}
 
 We first compute the subgroups $W$ of $\calC$ for which the determinant map $\calN/\bbar{H} \to Q_N$ gives an isomorphism $W\xrightarrow{\sim}Q_N$.    By Lemma~\ref{L:G(N) vs W},  the subgroups $G(N)$ of $\GL_2(\ZZ/N\ZZ)$ that satisfy the conditions (\ref{I:a}),  (\ref{I:b}) and (\ref{I:c}) of Definition~\ref{D:I(Gamma)} are precisely the inverse images of the groups $W$ under the quotient map $\calN \to \calC$.   We can then check condition (\ref{I:d}) for each of the groups $G(N)$.\\

Now fix one of the finite number of groups $G(N)$ that satisfies conditions (\ref{I:a}), (\ref{I:b}), (\ref{I:c}) and (\ref{I:d}) of Definition~\ref{D:I(Gamma)}.  Let $G$ be the inverse image of $G(N)$ under the reduction map $\GL_2(\ZZ_N) \to \GL_2(\ZZ/N\ZZ)$.    As usual, for an integer $M$ dividing some power of $N$, we let $G(M)$ be the image of $G$ in $\GL_2(\ZZ/M\ZZ)$; note that $G(N)$ agrees with the previous notation.

We shall now describe how to compute the index $[\SL_2(\ZZ_N) : G']$;  this is needed in order to compute $\scrI(\Gamma)$.   We remark that $G'(M)=G(M)'$.

\begin{lemma}\label{L:how to N2}
The group $G'$ contains $\{A \in \SL_2(\ZZ_N) : A\equiv I \pmod{N^2}\}$.  In particular, we have $[\SL_2(\ZZ_N): G'] = [\SL_2(\ZZ/N^2\ZZ): G(N^2)']$.
\end{lemma}
\begin{proof}
Since $G \supseteq I + N M_2(\ZZ_N)$, it suffices to prove that $(I+ N M_2(\ZZ_N))' = \SL_2(\ZZ_N) \cap (I + N^2 M_2(\ZZ_N))$.   So it suffices to prove that $(I+ q M_2(\ZZ_q))' = \SL_2(\ZZ_q) \cap (I + q^2 M_2(\ZZ_q))$ for any prime power $q>1$;  this is Lemma~1 of \cite{MR0568299}*{p.163}.
\end{proof}

Lemma~\ref{L:how to N2} allows us to compute $[\SL_2(\ZZ_N) : G']$ by computing the finite group $G(N^2)'$.    In practice, we will use the following to reduce the computation to finding $G(M)'$ for some, possibly smaller, divisor $M$ of $N^2$.

\begin{lemma} \label{L:commutator index step}
Let $r$ be the product of the primes dividing $N$.  Let $M>1$ be an integer having the same prime divisors as $N$.   If $G(rM)'$ contains $\{A \in \SL_2(\ZZ/rM\ZZ) : A \equiv I \pmod{M}\}$, then $[\SL_2(\ZZ_N): G'] = [\SL_2(\ZZ/M\ZZ): G(M)']$.
\end{lemma}
\begin{proof}
For each positive integer $m$, define the group $\calS_{m} := \{A \in \SL_2(\ZZ_m) : A\equiv I \pmod{m}\}$.

Let $H$ be a closed subgroup of $\SL_2(\ZZ_N)$ whose image in $\SL_2(\ZZ/rM\ZZ)$ contains $\{A \in \SL_2(\ZZ/rM\ZZ) : A \equiv I \pmod{M}\}$.   We claim that $H \supseteq  \calS_M$; the lemma will follow from the claim with $H=G'$.   By replacing $H$ with $H\cap\calS_M$, we may assume that $H$ is a closed subgroup of $\calS_M$.   Since $\calS_M$ is a product of the pro-$\ell$ groups $\calS_{\ell^{v_\ell(M)}}$ with $\ell | M$, we may further assume that $M$ is a power of a prime $\ell$ and hence $r=\ell$.\\

So fix a prime power $\ell^e>1$ and let $H$ be a closed subgroup of $\calS_{\ell^e}$ for which $H(\ell^{e+1}) = \{A \in \SL_2(\ZZ/\ell^{e+1}\ZZ): A \equiv I \pmod{\ell^e}\}$; we need to prove that $H = \calS_{\ell^e}$.    

For each integer $i\geq 1$, define $H_i := H \cap (I+\ell^i M_{2}(\ZZ_\ell))$ and $\mathfrak{h}_i := H_i/H_{i+1}$.   For any $A\in M_2(\ZZ_\ell)$ with $I +\ell^i A \in \SL_2(\ZZ_\ell)$, we have $\tr(A)\equiv 0\pmod{\ell}$.   The map $H_i \to M_2(\ZZ_\ell)$, $I+\ell^i A \mapsto A$ thus induces a homomorphism 
\[
\varphi_i\colon \mathfrak{h}_i \hookrightarrow \s_2(\FF_\ell),
\] 
where $\s_2(\FF_\ell)$ is the subgroup of trace $0$ matrices in $M_2(\FF_\ell)$.     Using that $H$ is closed, we deduce that $H=\calS_{\ell^e}$ if and only if $\varphi_i$ is surjective for all $i \geq e$.  

We now show that $\varphi_i$ is surjective for all $i \geq e$.   We proceed by induction on $i$; the homomorphism $\varphi_e$ is surjective by our initial assumption on $H$.  Now suppose that $\varphi_i$ is surjective for a fixed $i\geq e$.   Take any matrix $B$ in the set $\mathcal{B} := \{ \left(\begin{smallmatrix} 0 & 1  \\0 & 0\end{smallmatrix}\right), \left(\begin{smallmatrix} 0 & 0  \\1 & 0\end{smallmatrix}\right), \left(\begin{smallmatrix} 1 & 1  \\-1 & -1\end{smallmatrix}\right) \}$.     The matrix $I+\ell^i B$ has determinant $1$, so the surjectivity of $\varphi_i$ implies that there is a matrix $A \in M_2(\ZZ_\ell)$ with $A\equiv B \pmod{\ell}$ such that $h:= I + \ell^i A$ is an element of $H$.   

Working modulo $\ell^{2i+1}$, we find that $(\ell^i A)^2 = \ell^{2i} A^2 \equiv \ell^{2i} B^2 =0$, where the last equality uses that $B^2=0$.    In particular, $(\ell^i A)^2 \equiv 0 \pmod{\ell^{i+2}}$.   Therefore,
\[
h^{\ell} \equiv I + \tbinom{\ell}{1} \ell^{i} A  \equiv I +\ell^{i+1} A \equiv I +\ell^{i+1} B \pmod{\ell^{i+2}}.
\]
Since $h^\ell \in H$, we find that $B$ modulo $\ell$ lies in the image of $\varphi_{i+1}$.     Since $\s_2(\FF_\ell)$ is generated by the $B\in \mathcal{B}$, we deduce that $\varphi_{i+1}$ is surjective.
\end{proof}

\begin{remark}
In practice, a useful way to compute $G'$ is to first find open subgroups $B_\ell$ of $\GL_2(\ZZ_\ell)$ such that $\prod_{\ell |N} B_\ell \subseteq G$.  We can then compute $B_\ell' \subseteq \SL_2(\ZZ_\ell)$ using Lemma~\ref{L:commutator index step}.   Let $m_\ell$ the smallest power of $\ell$ for which $B_\ell$ is determined by its image modulo $m_\ell$.   We will then have $[\SL_2(\ZZ_N):G']=[\SL_2(\ZZ/M\ZZ): G(M)']$ where $M:=\prod_{\ell|N} m_\ell$.
\end{remark}

\subsection{Genus $0$ computations} \label{SS:genus 0 computations}
In this section, we compute the set of integers
\[
\scrI_0 := {\bigcup}_{\Gamma \in S_0} \scrI(\Gamma).
\]
Instead of computing $\scrI(\Gamma)$, we will compute two related quantities.   Let $\scrI'(\Gamma)$ be the set of integers as in Definition~\ref{D:I(Gamma)} but with condition (\ref{I:e}) excluded.  Let $\scrI''(\Gamma)$ be the set of integers as in Definition~\ref{D:I(Gamma)} with condition (\ref{I:e})  excluded and satisfying the additional condition that $X_{G(N)}^\infty(\QQ_p)$ is empty for at most one prime $p|N$.

\begin{lemma} \label{L:genus 0 inclusions}
For a congruence subgroup $\Gamma$ of genus $0$, we have $\scrI''(\Gamma) \subseteq \scrI(\Gamma) \subseteq \scrI'(\Gamma)$.
\end{lemma}
\begin{proof}
The inclusion $\scrI(\Gamma) \subseteq \scrI'(\Gamma)$ is obvious.   So assume that $G(N)$ is any group satisfying conditions (\ref{I:a})--(\ref{I:d}) of Definition~\ref{D:I(Gamma)} and that $X_{G(N)}^\infty(\QQ_p)$ is empty for at most one prime $p|N$.    To prove the inclusion $\scrI''(\Gamma) \subseteq \scrI(\Gamma)$, we need to verify that $X:=X_{G(N)}$ has infinitely many $\QQ$-points.  Note that the curve $X_\QQ$ is smooth and projective; it has genus $0$ by our assumption on $\Gamma$ and Proposition~\ref{P:isomorphic Riemann surfaces}.    

We claim that $X(\QQ_v)$ is non-empty for all places $v$ of $\QQ$; the places corresponds to the primes $p$ or to $\infty$ where $\QQ_\infty=\RR$.  Condition~(\ref{I:d}) and Proposition~\ref{P:R-points for YG} imply that $X(\RR)$ is non-empty.    Now take any prime $p\nmid N$.  As an $\ZZ[1/N]$-scheme $X$ has good reduction at $p$ and hence the fiber $X$ over $\FF_p$ is a smooth and projective curve of genus $0$.   Therefore, $X(\FF_p)$ is non-empty and any of the points can be lifted by Hensel's lemma to a point in $X(\QQ_p)$.  By our hypothesis on the sets $X_{G(N)}^\infty(\QQ_p)$ with $p|N$, we deduce that there is at most one prime $p_0$ such that $X(\QQ_{p_0})$ is empty.   

So suppose that there is precisely one prime $p_0$ for which $X(\QQ_{p_0})$ is empty.  The curve $X_\QQ$ has a model given by a conic of the form $ax^2+by^2-z^2=0$ with $a,b \in \QQ^\times$.  The \emph{Hilbert symbol} $(a,b)_v$, for a place $v$, is equal to $+1$ if $X(\QQ_v)\neq \emptyset$ and $-1$ otherwise.   Therefore, $\prod_v (a,b) = (a,b)_{p_0}=-1$.    However, we have $\prod_v (a,b) =1$ by reciprocity.   This contradiction proves our claim that $X(\QQ_v)$ is non-empty for all places $v$ of $\QQ$.

The curve $X_\QQ$ has genus $0$ so it satisfies the Hasse principle, and hence has a $\QQ$-rational point.  The curve $X_\QQ$ is thus isomorphic to $\PP^1_\QQ$ and has infinitely many $\QQ$-points.
\end{proof}

We shall use the explicit set $S_0$ due to Cummin and Pauli.   For each $\Gamma \in S_0$, it is straightforward to compute the set $\scrI'(\Gamma)$ using the method in \S\ref{SS:how to compute}.    

Using Lemma~\ref{L:cusp bijection} and the discussion in \S\ref{SS:how to compute}, we can also compute $\scrI''(\Gamma)$.  Fix a prime $p$ dividing $N$.  Take $e$ so that $p^e \parallel N$ and set $M = N/p^e$.   The image of the character $\chi _N \colon \Gal_{\QQ_p} \to (\ZZ/N\ZZ)^\times = (\ZZ/p^e\ZZ)^\times \times (\ZZ/M\ZZ)^\times$ arising from the Galois action on the $N$-th roots of unity is $(\ZZ/p^e\ZZ)^\times \times \langle p \rangle$.\\

Our \texttt{Magma} computations show that  $\bigcup_{\Gamma \in S_0} \scrI''(\Gamma) = \calI_0$ and $\bigcup_{\Gamma \in S_0} \scrI'(\Gamma) = \calI_0$, where 
\[
\calI_0 := \left\{\begin{array}{c}
 2, 4, 6, 8, 10, 12, 16, 20, 24, 30, 32, 36, 40, 48, 54, 60, 72, 84, 96, 108, 112,120, 144, \\192, 288, 336, 384, 576, 768, 864, 1152, 1200, 1296, 1536 
 \end{array}\right\}.
\]
Using the inclusions of Lemma~\ref{L:genus 0 inclusions}, we deduce that $\scrI_0 = \calI_0$.

\begin{remark}
From our genus $0$ computations, we find that $S_0$ has cardinality $121$ which led to $331$ total groups $G(N)$ that satisfied (\ref{I:a})--(\ref{I:d}) with respect to some $\Gamma \in S_0$.
\end{remark}

\subsection{Genus $1$ computations}  \label{SS:genus 1 computations}

Now define the set of integers
\[
\scrI_1 :=  {\bigcup}_{\Gamma \in S_1}\, (\scrI(\Gamma) - \calI_0),
\]
where $\calI_0$ is the set from \S\ref{SS:genus 0 computations}.

Instead of computing $\scrI(\Gamma)$, we will compute a related quantity.
We define $\scrI'''(\Gamma)$ to be the set of integers as in Definition~\ref{D:I(Gamma)} with condition (\ref{I:e})  excluded and satisfying the additional condition that the Mordell-Weil group of the Jacobian $J$ of the curve $X_{G(N)}$ over $\QQ$ has positive rank.   For a congruence subgroup $\Gamma$ of genus $1$, we have an inclusion $\scrI(\Gamma) \subseteq \scrI'''(\Gamma)$ since a genus $1$ curve over $\QQ$ that has a $\QQ$-point is isomorphic to its Jacobian.   Therefore,
\[
\scrI_1 \subseteq {\bigcup}_{\Gamma \in S_1}\, (\scrI'''(\Gamma) - \calI_0).
\]

We now explain how to compute $\scrI'''(\Gamma) - \calI_0$ for a fixed congruence subgroup $\Gamma$ of genus $1$.  As described in \S\ref{SS:how to compute}, we can compute the subgroups  $G(N)$ satisfying the conditions (\ref{I:a})--(\ref{I:d}).  For each group $G(N)$, it is described in \S\ref{SS:how to compute} how to compute $[\SL_2(\ZZ_N) : G']$, where $G$ is the inverse image of $G(N)$ under the reduction map $\GL_2(\ZZ_N) \to \GL_2(\ZZ/N\ZZ)$.   We may assume that $[\SL_2(\ZZ_N) : G'] \cdot 2/\gcd(2,N) \notin \calI_0$ since otherwise it does not contribute to $\scrI'''(\Gamma) - \calI_0$.

Let $J$ be the Jacobian of the curve $X_{G(N)}$ over $\QQ$; it is an elliptic curve since $\Gamma$ has genus $1$.     Let us now explain how to compute the rank of $J(\QQ)$ (and hence finish our method for computing $\scrI'''(\Gamma) - \calI_0$) without having to compute a model for $X_G$.    Moreover, we shall determine the elliptic curve $J$ up to isogeny (defined over $\QQ$); note that the Mordell rank is an isogeny invariant.\\

The curve $J$ has good reduction at all primes $p \nmid N$ since the $\ZZ[1/N]$-scheme $X_{G(N)}$ is smooth.    If $E/\QQ$ is an elliptic curve with good reduction at all primes $p\nmid N$, then its conductor divides $N_{\text{max}}:=\prod_{p | N} p^{e_p}$, where $e_2=8$, $e_3=5$ and $e_p = 2$ otherwise.          One can compute a finite list of elliptic curves 
\[
E_1,\ldots, E_n
\] 
over $\QQ$ that represent the isogeny classes of elliptic curves over $\QQ$ with good reduction at $p\nmid N$.   In our computations, we will have $N_{\text{max}} \leq 2^8 \cdot 3^5 = 62208$ and hence the representative curves $E_i$ can all be found in Cremona's database \cite{cremona} of elliptic curves which are included in \texttt{Magma}  (it currently contains all elliptic curves over $\QQ$ with conductor at most $500000$).   It remains to determine which curve $E_i$ is isogenous to $J$.

Take any prime $p \nmid N$.   Using the methods of \S\ref{SS:Fp points}, we can compute the cardinality of $X_{G(N)}(\FF_p)$ and hence also the \emph{trace of Frobenius} 
\[
a_p(J)= p+1 - |J(\FF_p)|= p+1 - |X_{G(N)}(\FF_p)|.
\]  
If $a_p(E_i) \neq a_p(J)$, then $E_i$ and $J$ are not isogenous elliptic curves over $\QQ$.   By computing $a_p(J)$ for enough primes $p\nmid N$, one can eventually eliminate all but one curve $E_{i_0}$ which then must be isogenous to $J$.   There are then known methods to determine the Mordell rank of $E_{i_0}$; the rank is also part of Cremona's database.  Therefore, we can compute the rank of $J(\QQ)$.
\\

Our \texttt{Magma} computations show that
\[
\bigcup_{\Gamma \in S_1} (\scrI'''(\Gamma) - \calI_0) = \{220, 240, 360, 504\}.
\]
In particular, $\scrI_1 \subseteq \{220,240,360,504\}$.\\

We now describe how the values $220$, $240$, $360$ and $504$ arise in our computations.    

For an odd prime $\ell$, let $\calN_{\ell}^-$ be the normalizer in $\GL_2(\ZZ/\ell\ZZ)$ of a non-split Cartan subgroup and let $\calN_{\ell}^+$ be the normalizer in $\GL_2(\ZZ/\ell\ZZ)$ of a split Cartan subgroup.     Define $G_1:= \calN_{11}^-$.   We can identify $\calN_3^- \times \calN_5^-$ and $\calN_3^- \times \calN_5^+$ with subgroups $G_2$ and $G_3$, respectively, of $\GL_2(\ZZ/15\ZZ)$.   We can identify $\calN_3^- \times \calN_7^-$ with a subgroup $G_4$ of $\GL_2(\ZZ/21\ZZ)$.

Fix an $n \in \{220, 240, 360, 504\}$.   Let $\Gamma \in S_1$ be any congruence subgroup such that $n \in \scrI(\Gamma)$.    
Let $G(N)$ be one of the groups such that the following hold:
\begin{itemize}
\item it satisfies conditions (\ref{I:a}),  (\ref{I:b}), (\ref{I:c}) and (\ref{I:d}) of Definition~\ref{D:I(Gamma)},
\item the Jacobian $J$ of the curve $X_{G(N)}$ over $\QQ$ has positive rank,
\item we have $[\SL_2(\ZZ_N) : G'] \cdot 2/\gcd(2,N) = n$, where $G$ is the inverse image of $G(N)$ under the reduction $\GL_2(\ZZ_N)\to \GL_2(\ZZ/N\ZZ)$.
\end{itemize}
Our computations show that one of the following hold:
\begin{itemize}
\item
We have $n=220$, $N = 11$ and $G(N)$ is conjugate in $\GL_2(\ZZ/11\ZZ)$ to $G_1$.
\item
We have $n=240$, $N = 15$ and $G(N)$ is conjugate in $\GL_2(\ZZ/15\ZZ)$ to $G_2$.
\item
We have $n=360$, $N = 15$ and $G(N)$ is conjugate in $\GL_2(\ZZ/15\ZZ)$ to $G_3$.
\item
We have $n=504$, $N = 21$ and $G(N)$ is conjugate in $\GL_2(\ZZ/21\ZZ)$ to $G_4$.
\end{itemize}
For later, we note that the index $[\GL_2(\ZZ/N\ZZ) : G_i]$ is $55$, $30$, $45$ or $63$ for $i=1$, $2$, $3$ or $4$, respectively.

\begin{lemma}  \label{L:scrI1}
We have $\scrI_1 = \{220, 240, 360, 504\}$.
\end{lemma}
\begin{proof}
We already know the inclusion $\scrI_1  \subseteq \{220, 240, 360, 504\}$.   It thus suffices to show that the set $X_{G_i}(\QQ)$ is infinite for all $1\leq i\leq 4$.   So for a fixed $i\in \{1,2,3,4\}$, it suffices to show that $X_{G_i}(\QQ)$ is non-empty, since it then becomes isomorphic to its Jacobian which we know has infinitely many rational points.  By Proposition~\ref{P:key}, it suffices to find a single elliptic curve $E/\QQ$ with $j_E \notin \{0,1728\}$ for which $\rho_{E,N}(\Gal_\QQ)$ is conjugate to a subgroup of $G_i$.

Let $E/\QQ$ be a CM elliptic curve.   Define $R:= \End(E_{\Qbar})$; it is an order in the imaginary quadratic field $K:=R\otimes_\ZZ \QQ$.  Take any odd prime $\ell$ that does not divide the discriminant of $R$.   One can show that $\rho_{E,\ell}(\Gal_\QQ)$ is contained in the normalizer of a Cartan subgroup $C \subseteq \GL_2(\ZZ/\ell\ZZ)$ isomorphic to $(R/\ell R)^\times$, cf.~\cite{MR1757192}*{Appendix~A.5}.    The Cartan group $C$ is split if and only if $\ell$ splits in $K$.

Consider the CM curve $E_1/\QQ$ defined by $y^2=x^3-11x+14$; $R$ is an order in $\QQ(i)$ of discriminant $-16$.     The primes $3$, $7$ and $11$ are inert in $\QQ(i)$ and $5$ is split in $\QQ(i)$.     Therefore, $\rho_{E_1,11}(\Gal_\QQ)$, $\rho_{E_1,15}(\Gal_\QQ)$ and $\rho_{E_1,21}(\Gal_\QQ)$ are conjugate to subgroups of $G_1$, $G_3$ and $G_4$, respectively.

Consider the CM curve $E_2/\QQ$ defined by $y^2+xy = x^3-x^2-2x-1$; $R$ is an order in $\QQ(\sqrt{-7})$ of discriminant $-7$.     The primes $3$ and $5$ are inert in $\QQ(\sqrt{-7})$.     Therefore, $\rho_{E_2,15}(\Gal_\QQ)$ is conjugate to a subgroup of $G_2$.
\end{proof}

\begin{remark}
From our genus $1$ computations, we find that $S_1$ has cardinality $163$ which led to $805$ total groups $G(N)$ that satisfied (\ref{I:a})--(\ref{I:d}) with respect to some $\Gamma \in S_1$.   We needed to determine the Jacobian of $X_{G(N)}$, up to isogeny, for $63$ of these groups $G(N)$.
\end{remark}

\subsection{Proof of Proposition~\ref{P:calI computation}}
In \S\ref{SS:genus 0 computations}, we found that $\bigcup_{\Gamma\in S_0} \scrI(\Gamma) = \calI_0$.  By Lemma~\ref{L:scrI1}, we have 
\[
\Big(\bigcup_{\Gamma\in S_1} \scrI(\Gamma) \Big) - \calI_0 =  \bigcup_{\Gamma\in S_1} (\scrI(\Gamma) - \calI_0) = \{220, 240, 360, 504\}. 
\]
Therefore, $\scrI$ is equal to $\calI_0 \cup \{220,240,360,504\} =\calI$.

\section{Proof of main theorems}  \label{S:proofs}

\subsection{Proof of Theorem~\ref{T:main ineffective}}
The theorem follows immediately from Theorem~\ref{T:main ineffective scrI} and Proposition~\ref{P:calI computation}.

\subsection{Proof of Theorem~\ref{T:index Q}}

\begin{lemma}  \label{L:easy connection}
Let $E/\QQ$ be a non-CM elliptic curve and suppose $\ell > 37$ is a prime for which $\rho_{E,\ell}$ is not surjective.   Then $\ell \leq [\GL_2(\Zhat) :\rho_E(\Gal_\QQ)]$.
\end{lemma}
\begin{proof}
  From \cite{MR644559}*{\S8.4}, we find that $\rho_{E,\ell}(\Gal_\QQ)$ is contained in the normalizer of a Cartan subgroup of $\GL_2(\ZZ/\ell\ZZ)$.   In particular, we have $[\GL_2(\ZZ/\ell\ZZ) : \rho_{E,\ell}(\Gal_\QQ) ] \geq \ell (\ell-1)/2 \geq \ell$.   Therefore, 
  $\ell \leq [\GL_2(\ZZ/\ell\ZZ) : \rho_{E,\ell}(\Gal_\QQ) ] \leq [\GL_2(\Zhat): \rho_E(\Gal_\QQ)]$.
\end{proof}

First suppose that there is a finite set $J$ such that if $E/\QQ$ is an elliptic curve with $j_E\notin J$, then $[\GL_2(\Zhat): \rho_E(\Gal_\QQ)] \in \calI$.  There is thus an integer $c>37$ such that for any non-CM $E/\QQ$, we have $[\GL_2(\Zhat): \rho_E(\Gal_\QQ)] \leq c$, this uses Serre's theorem (and Lemma~\ref{C:index depends on j}) to deal with the finite number of $j$-invariants of non-CM curves that are in $J$.  By Lemma~\ref{L:easy connection}, we deduce that $\rho_{E,\ell}$ is surjective for all primes $\ell > c$; this gives Conjecture~\ref{C:Serre1}.\\

Now suppose that Conjecture~\ref{C:Serre1} holds for some constant $c$.    Let $J$ be the finite set from Theorem~\ref{T:main ineffective} with this constant $c$.  After possibly increasing $J$, we may assume that it contains the finite number of $j$-invariants of CM elliptic curves over $\QQ$.     Theorem~\ref{T:main ineffective} then implies that for any elliptic curve $E/\QQ$ with $j_E \notin J$, we have $[\GL_2(\Zhat): \rho_{E}(\Gal_\QQ)] \in \calI$.

\subsection{Proof of Theorem~\ref{T:Jn}}
First take any $n\geq 1$ so that $J_n$ is infinite.  Let $E/\QQ$ be an elliptic curve with $j_E \in J_n$, equivalently, with $[\GL_2(\Zhat): \rho_E(\Gal_\QQ)] =n$.  Lemma~\ref{L:easy connection} implies that $\rho_{E,\ell}$ is surjective for all primes $\ell > \max\{37,n\}$.  Let $J$ be the set from Theorem~\ref{T:main ineffective} with $c:=\max\{37,n\}$.   Now take any elliptic curve $E/\QQ$ with $j_E \in J_n - J$; note that $J_n-J$ is non-empty since $J_n$ is infinite and $J$ is finite.   The representation $\rho_{E,\ell}$ is surjective for all $\ell > c$  and $j_E \notin J$, so $[\GL_2(\Zhat): \rho_E(\Gal_\QQ)] $ is an element of $\calI$ by Theorem~\ref{T:main ineffective}.   Therefore, $n\in \calI$.\\

Now take any integer $n \in \calI$.  To complete the proof of the theorem,  we need to show that $J_n$ is infinite.    By Proposition~\ref{P:calI computation}, we have $n \in \scrI(\Gamma)$ for some congruence subgroup $\Gamma$ of $\SL_2(\ZZ)$ of genus $0$ or $1$.    From our computation of $\scrI_0$ in \S\ref{SS:genus 0 computations}, we may assume that $\Gamma$ has genus $0$ when $n \notin \{220, 240, 360, 504\}$.   

Denote the level of $\Gamma$ by $N_0$.  Let $N$ be the integer $N_0$, $4N_0$ or $2N_0$ when $v_2(N_0)$ is $0$, $1$ or at least $2$, respectively.     The integer $N$ is not divisible by any prime $\ell > 13$ (if $\Gamma$ has genus $0$, this follows from the classification of  genus $0$ congruence subgroups in  \cite{MR2016709}; if $\Gamma$ has genus $1$, then we saw in \S\ref{SS:genus 1 computations} that $N\in\{11,15,21\}$).\\

Since $n\in \scrI(\Gamma)$, there is a subgroup $G(N)$ of $\GL_2(\ZZ/N\ZZ)$ that satisfies conditions (\ref{I:a}), (\ref{I:b}), (\ref{I:c}), (\ref{I:d}) and (\ref{I:e}) of Definition~\ref{D:I(Gamma)} and also satisfies $n=[\SL_2(\ZZ_N) : G_N'] \cdot 2/\gcd(2,N)$, where $G_N$ is the inverse image of $G(N)$ under the reduction map $\GL_2(\ZZ_N) \to \GL_2(\ZZ/N\ZZ)$.   Let $G$ be the inverse image of $G(N)$ under $\GL_2(\Zhat)\to \GL_2(\ZZ/N\ZZ)$. \\

Let $m$ be the product of the primes $\ell \leq 13$; note that $N$ divides some power of $m$.   Let $G_m$ be the image of $G$ under the projection map $\GL_2(\Zhat) \to \GL_2(\ZZ_m)$.  Lemma~\ref{L:Frattini} implies that there is a positive integer $M$, dividing some power of $m$, such that if $H$ is an open subgroup of $G_m \subseteq\GL_2(\ZZ_m)$, then $H$ equals $G_m$ if and only if $H(M)$ equals $G_m(M)=G(M)$.  

Take any proper subgroup $B \subseteq  G(M)$ for which $\det(B)=(\ZZ/M\ZZ)^\times$ and $-I \in B$.    We have a morphism $\varphi_B \colon Y_B \to Y_{G(M)} = Y_{G(N)}$ of curves over $\QQ$ such that $\pi_B = \pi_{G(N)} \circ \varphi_B$.    The morphism $\varphi_B$ has degree $[G(M): B] >1$.  Define
\[
W := \bigcup_B \varphi_B( Y_B(\QQ)),
\]
where $B$ varies over the proper subgroups of $G(M)$ for which $\det(B)=(\ZZ/M\ZZ)^\times$ and $-I \in B$.   We have $W \subseteq Y_{G(N)}(\QQ)$.

\begin{lemma} \label{L:mod M details}
If $E/\QQ$ is a non-CM elliptic curve with $j_E \in \pi_{G(N)}(Y_{G(N)}(\QQ) - W)$, then $\pm \rho_{E,M}(\Gal_\QQ)$ is conjugate in $\GL_2(\ZZ/M\ZZ)$ to $G(M)$.
\end{lemma}
\begin{proof}
Fix a non-CM elliptic curve $E/\QQ$ with $j_E \in \pi_{G(N)}(Y_{G(N)}(\QQ) - W)= \pi_{G(M)}(Y_{G(M)}(\QQ) - W)$.  There is a point $P \in Y_G(\QQ)-W$ for which $\pi_{G(M)}(P)=j_E$.   

With notation as in \S\ref{S:modular curves}, there is an isomorphism $\alpha\colon E[M]\xrightarrow{\sim} (\ZZ/M\ZZ)^2$ such that the pair $(E,[\alpha]_G)$ represents $P$.   Since $j_E \notin \{0,1728\}$, the automorphisms of $E_{\Qbar}$ act on $E[N]$ by $I$ or $-I$.    By Lemma~\ref{L:YG(k) meaning}(\ref{L:YG(k) meaning ii}) and $-I \in G(M)$,  we have $\alpha \circ \sigma^{-1} \circ \alpha^{-1} \in G(M)$ for all $\sigma\in \Gal_\QQ$.   We may assume that $\rho_{E,M}$ was chosen so that $\rho_{E,M}(\sigma)=\alpha \circ \sigma \circ \alpha^{-1}$ for all $\sigma\in \Gal_\QQ$.   Since $-I\in G(M)$, we deduce that $B:=\pm \rho_{E,M}(\Gal_\QQ)$ is a subgroup of $G(M)$.  Note that $\det(B)=(\ZZ/M\ZZ)^\times$ and $-I \in B$.

Suppose that $B$ is a proper subgroup of $G(M)$.   We have $\alpha \circ \sigma^{-1} \circ \alpha^{-1} \in B$ for all $\sigma\in \Gal_\QQ$, so $(E,[\alpha]_B)$ represents a point $P'\in Y_B(\QQ)$ by  Lemma~\ref{L:YG(k) meaning}(\ref{L:YG(k) meaning ii}).     We have $\varphi_B(P') = P$, so $P \in W$.   This contradict that $P\in Y_G(\QQ)-W$ and hence $B=G(M)$.
\end{proof}

\begin{lemma} \label{L:W set}
If $E/\QQ$ is an elliptic curve with $j_E \in \pi_{G(N)}(Y_{G(N)}(\QQ) - W)$, then 
\[
[\GL_2(\Zhat): \rho_E(\Gal_\QQ) ] = n
\]
 or $\rho_{E,\ell}$ is not surjective for some prime $\ell > 13$.  

\end{lemma}
\begin{proof}
Let $E/\QQ$ be an elliptic curve with $j_E \in \pi_{G(N)}(Y_{G(N)}(\QQ) - W)$ such that $\rho_{E,\ell}$ is surjective for all $\ell>13$.  We need to show that $[\GL_2(\Zhat): \rho_E(\Gal_\QQ) ] = n$.   The curve $E$ is non-CM since $\rho_{E,\ell}$ is surjective for $\ell>13$.  Define the subgroup 
\[
H := \Zhat^\times \cdot \rho_E(\Gal_\QQ)
\]
of $\GL_2(\Zhat)$.    By Lemma~\ref{L:mod M details}, we may assume that $\pm \rho_{E,M}(\Gal_\QQ) = G(M)$.   Since $G(M)$ contains the scalar matrices in $\GL_2(\ZZ/M\ZZ)$, we have $H(M)=G(M)$ and an inclusion $H \subseteq G$.  In particular, $H' \subseteq G'$.   

Let $m_0$ be the product of the primes $\ell$ for which $\ell \leq 5$ or for which $\rho_{E,\ell}$ is not surjective.     Let $H_m$ and $H_{m_0}$ be the image of $H$ under the projection to $\GL_2(\ZZ_m)$ and $\GL_2(\ZZ_{m_0})$, respectively.    The integer $m_0$ divides $m$ since $\rho_{E,\ell}$ is surjective for all $\ell>13$.

 Lemma~\ref{L:technical group section} applied with $G$ and $m$ replaced by $H$ and $m_0$, respectively, implies that  $H' = H_{m_0} ' \times {\prod}_{\ell \nmid m_0 } \SL_2(\ZZ_\ell)$.   Therefore, we have 
\[
H' = H_m ' \times {\prod}_{\ell \nmid m } \SL_2(\ZZ_\ell).
\]
Since $H' \subseteq G' \subseteq \SL_2(\Zhat)$, we deduce that 
\[
G' = G_m' \times {\prod}_{\ell \nmid m } \SL_2(\ZZ_\ell).
\]    
We have $H_m \subseteq G_m$ and $H(M)=G(M)$, and thus $H_m=G_m$ by our choice of $M$.    Therefore, $H_m'=G_m'$ and hence $H' = G'$.  The groups $H$ and $\rho_E(\Gal_\QQ)$ have the same commutator subgroup, so by Proposition~\ref{P:EC equality}, we have
\[
[\GL_2(\Zhat):\rho_E(\Gal_\QQ)]  = [\SL_2(\Zhat)\colon H']  = [\SL_2(\Zhat)\colon G'].
\]
It remains to show that $[\SL_2(\Zhat)\colon G'] = n$.   We have $G = G_N \times \prod_{\ell \nmid N} \GL_2(\ZZ_\ell)$, so $G' = G_N' \times \prod_{\ell \nmid N} \GL_2(\ZZ_\ell)'$.   By Lemma~\ref{L:SL2 derived group theory}, the index $[\SL_2(\ZZ_\ell) : \GL_2(\ZZ_\ell)']$ is $1$ or $2$ when $\ell\neq 2$ or $\ell=2$, respectively.   Therefore, 
\[
[\SL_2(\Zhat): G ' ] = [\SL_2(\ZZ_N) : G_N'] \cdot \prod_{\ell \nmid N} [\SL_2(\ZZ_\ell) : \GL_2(\ZZ_\ell)'] = [\SL_2(\ZZ_N) : G_N'] \cdot 2/\gcd(2,N) = n.  \qedhere
\]
\end{proof}

Recall that a subset $S$ of $\PP^1(\QQ)$ has \emph{density} $\delta$ if 
\[
|\{P \in S: h(P) \leq x\}|/|\{P \in \PP^1(\QQ): h(P) \leq x\}| \to \delta
\] 
as $x\to \infty$, where $h$ is the height function.   If $X_{G(N)}$ has genus $0$, then it is isomorphic to $\PP^1_\QQ$ (from our assumptions on $G(N)$, the curve $X_{G(N)}$ has infinitely many $\QQ$-points).   Choosing such an isomorphism $X_{G(N)} \cong \PP^1_\QQ$ allows us to define the density of a subset of $X_{G(N)}(\QQ)$; the existence and value of the density does not depend on the choice of isomorphism.

\begin{lemma} \label{L:selective surjectivity}
There is an infinite subset $S$ of $Y_{G(N)}(\QQ)$, with positive density if $X_{G(N)}$ has genus $0$, such that if $E/\QQ$ is an elliptic curve with $j_E \in \pi_{G(N)}(S)$, then $\rho_{E,\ell}$ is surjective for all $\ell > 13$.
\end{lemma}
\begin{proof}
We claim that for any place $v$ of $\QQ$, the set $X_{G(N)}(\QQ)$ has no isolated points in $X_{G(N)}(\QQ_v)$, i.e., there is no open subset $U$ of $X_{G(N)}(\QQ_v)$, with respect to the $v$-adic topology, for which $U \cap X_{G(N)}(\QQ)$ consists of a single point.    If $X_{G(N)}$ has genus $0$, then the claim follows since no point in $\PP^1(\QQ)$ is isolated in $\PP^1(\QQ_v)$.  Now consider the case where $X_{G(N)}$ has genus $1$.  If one point of $X_{G(N)}(\QQ)$ was isolated in $X_{G(N)}(\QQ_v)$, then using the group law of $X_{G(N)}(\QQ)$ (by first fixing a rational point), we find that every point is isolated.   So suppose that for each $P \in X_{G(N)}(\QQ)$, there is an open subset $U_P \subseteq X_{G(N)}(\QQ_v)$ such that $U_P \cap X_{G(N)}(\QQ) = \{P\}$.    The sets $\{U_P\}_{P\in X_{G(N)}(\QQ)}$ along with the complement of the closure of $X_{G(N)}(\QQ)$ in $X_{G(N)}(\QQ_v)$ form an open cover of $X_{G(N)}(\QQ_v)$ that has no finite subcover.  This contradicts the compactness of $X_{G(N)}(\QQ_v)$  and proves the claim.

Since $\pi_{G(N)} \colon Y_{G(N)}(\RR) \to \RR$ is continuous, the above claim with $v=\infty$ implies that the set $\pi_{G(N)}(Y_{G(N)}(\QQ))$ is not a subset of $\ZZ$.     Choose a rational number $j\in \pi_{G(N)}(Y_{G(N)}(\QQ))$ that is \emph{not} an integer.   

There is a prime $p$ such that $v_p(j)$ is negative; set $e:=-v_p(j)$.   Let $\calU$ be the set of points $P \in Y_{G(N)}(\QQ_p)$ for which $\pi_{G(N)}(P) \neq 0$ and $v_p(\pi_{G(N)}(P)) = -e$; it is an open subset of $Y_{G(N)}(\QQ_p)$.   Define $S := \calU \cap Y_{G(N)}(\QQ)=\calU \cap X_{G(N)}(\QQ)$; it is non-empty by our choice of $e$ (in particular, $\calU$ is non-empty).  The set $S$ is infinite since otherwise there would be an isolated point of $X_{G(N)}(\QQ)$ in $X_{G(N)}(\QQ_p)$.     If $X_{G(N)}$ has genus $0$,  then $S$ clearly has positive density.\\

Now take any elliptic curve $E/\QQ$ with $j_E \in \pi_{G(N)}(S)$ and any prime $\ell > \max\{ 37, e\}$; it is non-CM since its $j$-invariant is not an integer.   We claim that $\rho_{E,\ell}$ is surjective.   The lemma will follow from the claim after using Proposition~\ref{P:finiteness ell} to remove a finite subset from $S$ to ensure the surjectivity of $\rho_{E,\ell}$ for $13<\ell \leq \max\{37, e\}$.

Suppose that $\rho_{E,\ell}$ is not surjective.   From Lemmas 16, 17 and 18 in \cite{MR644559}, we find that $\rho_{E,\ell}(\Gal_\QQ)$ is  contained in the normalizer of a Cartan subgroup of $\GL_2(\ZZ/\ell\ZZ)$.   In particular, the order of $\rho_{E,\ell}(\Gal_\QQ)$ is not divisible by $\ell$.

We have $v_p(j_E)=-e < 0$ since $j_E \in \pi_{G(N)}(S)$.  Let $E'/\QQ_p$ be the \emph{Tate curve} with $j$-invariant $j_E$; see \cite{MR1484415}*{IV~Appendix~A.1} for details.   From the proposition in \cite{MR1484415}*{IV~Appendix~A.1.5} and our assumption $\ell > e$, we find that $\rho_{E',\ell}(\Gal_{\QQ_p})$ contains an element of order $\ell$.    Since $E'$ and $E$ have the same $j$-invariant, they become isomorphic over some quadratic extension of $\QQ_p$.   Since $\ell$ is odd, we deduce that $\rho_{E,\ell}(\Gal_\QQ)$ contains an element of order $\ell$.       This contradicts that the order of $\rho_{E,\ell}(\Gal_\QQ)$ is not divisible by $\ell$.  Therefore, $\rho_{E,\ell}$ is surjective as claimed.
\end{proof}

Let $W$ and $S$ be the sets from Lemma~\ref{L:W set} and Lemma~\ref{L:selective surjectivity}, respectively.    Take any elliptic curve $E/\QQ$ with $j_E \in \pi_{G(N)}(S-W)$.   Lemma~\ref{L:selective surjectivity} implies that the representation $\rho_{E,\ell}$ is surjective for all $\ell >13$.    Lemma~\ref{L:W set} then implies that $[\GL_2(\Zhat): \rho_E(\Gal_\QQ)] = n$.    Therefore, $J_n \supseteq  \pi_{G(N)}(S-W)$.      So to prove that $J_n$ is infinite, it suffices to show that the set $S-W$ is infinite.

First suppose that $X_{G(N)}$ has genus $0$.   The set $W$ is a \emph{thin} subset of $X_{G(N)}(\QQ) \cong \PP^1(\QQ)$ in the language of \cite{MR1757192}*{\S9.1};  this uses that the union defining $W$ is finite and that the morphisms $\varphi_B$ are dominant with degree at least $2$.    From \cite{MR1757192}*{\S9.7}, we find that $W$ has density $0$.    Since $S$ has positive density, we deduce that $S-W$ is infinite.

Finally suppose that $X_{G(N)}$ has genus $1$.   Since $S$ is infinite, it suffices to show that $W$ is finite.    So take any proper subgroup $B$ of $G(M)$ satisfying $\det(B)=(\ZZ/M\ZZ)^\times$ and $- I \in B$.     It thus suffices to show that the set $X_B(\QQ)$ is finite.    The morphism $\varphi_B\colon X_B \to X_{G(N)}$ is dominant, so $X_B$ has genus at least $1$.   If $X_B$ has genus greater than $1$, then $X_B(\QQ)$ is finite by Faltings' theorem.    We are left to consider the case where $X_B$ has genus $1$.   Let $\Gamma_B$ be the congruence subgroup associated to $X_B$; it has genus $1$.   We have $\Gamma_B \subseteq \Gamma$ and hence the level of $\Gamma_B$ is divisible by $N_0$.  We have $[\SL_2(\ZZ):\Gamma_B] = [\GL_2(\ZZ/M\ZZ) : B]$ and hence $b:=[\GL_2(\ZZ/M\ZZ) : G(M)]=[\GL_2(\ZZ/N\ZZ): G(N)]$ is a proper divisor of $[\SL_2(\ZZ):\Gamma_B]$.    From the computations in \S\ref{SS:genus 1 computations}, we  may assume that $G(N)$ is equal to one of the groups denoted $G_1$, $G_2$, $G_3$ or $G_4$.   In particular, we have $(N_0,b) \in \{(11,55) , (15,30), (15,45), (21,63)\}$.     From the classification in \cite{MR2016709}, we find that there are no genus $1$ congruence subgroups of $\SL_2(\ZZ)$ containing $-I$ whose level is divisible by $N_0$ and whose index in $\SL_2(\ZZ)$ has $b$ as a proper divisor.   So the case where $X_B$ has genus $1$ does not occur and we are done.



\end{document}